\documentclass[letterpaper, 10pt]{article}
\usepackage{amsmath, amsthm, amssymb, amsfonts}

\newtheorem{thm}{Theorem}[section]
\newtheorem{our_thm}{Theorem}

\newtheorem{lem}[thm]{Lemma}
\newtheorem{prop}[thm]{Proposition}

\newtheorem{clm}[thm]{Claim}
\theoremstyle{definition}
\newtheorem{defn}[thm]{Definition}
\theoremstyle{remark}
\newtheorem{rem}[thm]{Remark}

\newcommand{\sExp}[2]{\mathbf{E}_{#1}\left[#2\right]}
\newcommand{\Exp}[1]{\sExp{}{#1}}
\newcommand{\Prob}[1]{\Pr\left[#1\right]}
\newcommand{\cProb}[3]{\Pr_{#1}\left[ \left. #2 \;\right\vert #3 \right]}

\newcommand{\mA}{\mathcal{A}}
\newcommand{\mB}{\mathcal{B}}
\newcommand{\mC}{\mathcal{C}}
\newcommand{\mF}{\mathcal{F}}
\newcommand{\mP}{\mathcal{P}}
\newcommand{\mG}{\mathcal{G}}
\newcommand{\mL}{\mathcal{L}}
\newcommand{\mX}{\mathcal{X}}
\newcommand{\mD}{\mathcal{D}}
\newcommand{\mM}{\mathcal{M}}
\newcommand{\mS}{\mathcal{S}}
\newcommand{\hG}{\widehat{G}}
\newcommand{\hH}{\widehat{G}}

\newcommand{\Gnp}[2]{\mG(#1,#2)}

\newcommand{\GNp}[1]{\Gnp{n}{#1}}
\newcommand{\GNP}{\GNp{p}}

\newcommand{\GNM}{\GNp{M}}
\newcommand{\Bin}{\textrm{Bin}}
\newcommand{\tG}{\widetilde{G}}
\newcommand{\PerfMatch}{\mathcal{PM}}
\newcommand{\Hamiltonicity}{\mathcal{HAM}}
\newcommand{\EdgeConn}[1]{\mathcal{EC}_{#1}}
\newcommand{\VertConn}[1]{\mathcal{VC}_{#1}}

\newcommand{\Small}{\textsc{Small}}

\title{Hitting time results for Maker-Breaker games}
\author{
Sonny Ben-Shimon
\thanks{The Blavatnik School of Computer Science, Raymond and Beverly Sackler Faculty of Exact Sciences, Tel Aviv University, 69978, Israel. Email: sonny@post.tau.ac.il. Research partially supported by a Farajun Foundation Fellowship.}
\and Asaf Ferber
\thanks{School of Mathematical Sciences, Raymond and Beverly Sackler Faculty of Exact Sciences, Tel Aviv University, 69978, Israel. Email: ferberas@post.tau.ac.il}
\and Dan Hefetz
\thanks{Institute of Theoretical Computer Science, ETH Z\"urich, CH-8092 Switzerland. Email: dan.hefetz@inf.ethz.ch.}
\and Michael Krivelevich
\thanks{School of Mathematical Sciences, Raymond and Beverly Sackler Faculty of Exact Sciences, Tel Aviv University, Tel Aviv 69978, Israel. Email: krivelev@post.tau.ac.il. Research supported in part by USA-Israel BSF grant 2006322 and by grant 1063/08 from the Israel Science Foundation.}
}

\begin{document}
\maketitle

\begin{abstract}
We study Maker-Breaker games played on the edge set of a random graph. Specifically, we consider the random graph process and analyze the first time in a typical random graph process that Maker starts having a winning strategy for his final graph to admit some property $\mP$. We focus on three natural properties for Maker's graph, namely being $k$-vertex-connected, admitting a perfect matching, and being Hamiltonian. We prove the following optimal hitting time results: with high probability Maker wins the $k$-vertex connectivity game exactly at the time the random graph process first reaches minimum degree $2k$; with high probability Maker wins the perfect matching game exactly at the time the random graph process first reaches minimum degree $2$; with high probability Maker wins the Hamiltonicity game exactly at the time the random graph process first reaches minimum degree $4$. The latter two statements settle conjectures of Stojakovi\'{c} and Szab\'{o}.
\end{abstract}

\section{Introduction}
Let $X$ be a finite set and let ${\mF} \subseteq 2^X$ be a family of subsets. In the positional game $(X,{\mF})$, two players take turns in claiming one previously unclaimed element of $X$ and the game ends when all of the elements of $X$ have been claimed by either of the players. The set $X$ is often referred to as the \emph{board} of the game. Positional games have attracted a lot of attention in the past decade and a thorough introduction to this field with a plethora of results can be found in a recent monograph of Beck \cite{Bec2008}. In a {\em Maker-Breaker}-type positional game, the two players are called \emph{Maker} and \emph{Breaker} and the members of ${\mF}$ are referred to as the \emph{winning sets}. Maker wins the game if he occupies all elements of some winning set; otherwise Breaker wins. We will always assume that Breaker starts the game. We say that a game $(X,{\mF})$ is a {\em Maker's win} if Maker has a strategy (that can be adaptive to Breaker's moves) that ensures his win in this game against any strategy of Breaker, otherwise the game is a {\em Breaker's win}. Note that $X$ and ${\mF}$ alone determine whether the game is a Maker's win or a Breaker's win. A classical example of this Maker-Breaker setting is the popular board game HEX.

\subsection{Maker-Breaker games on graphs}
Let $G=(V,E)$ be a graph and let $\mP$ be a monotone increasing graph property on $V$ (a family of graphs on $V$, closed under isomorphism and addition of edges). Consider the Maker-Breaker game $(E, \mF_{\mP})$ played on the edge set $E$ as the board of the game. The game is a win for Maker if and only if the graph spanned by the edges selected by Maker throughout the game satisfies the property $\mP$. We denote the family of graphs $G$ for which the $(E(G), \mF_{\mP})$ game is a Maker's win by $\mM_{\mP}$. Although the above game is described in game-theoretic terms, it should be noted that these games are finite perfect information games with no chance moves, and $\mM_{\mP}$ is some graph property which clearly satisfies $\mM_{\mP}\subseteq \mP$. Moreover, since $\mP$ is monotone increasing, $\mM_{\mP}$ is clearly monotone increasing as well. By considering monotone increasing graph properties, the game can be terminated as soon as the graph spanned by Maker's edges satisfies the property, regardless of whether all edges have been claimed or not. This leads to several natural questions. First, how sparse can a graph $G\in\mM_{\mP}$ be? In this context, playing on random graphs (where the density of the graph is chosen according to the property at hand) becomes very natural. This setting was formally initiated in \cite{StoSza2005} by Stojakovi\'{c} and Szab\'{o}, and this current work is a further exploration of it. Second, one can also study the minimum number of moves needed for Maker in order to win the game (see e.g. \cite{Bec81, Pek96, FelKri2008, HefEtAl2009b, HefSti2009}), but ``winning fast'' is not in the focus of this current work.

\subsection{Random graphs}\label{ss:RandomGraphs}
The most widely used random graph model is the Binomial random graph, $\GNP$. In this model we start with $n$ vertices, labeled, say, by $V=\{1,\ldots,n\}=[n]$, and select a graph on these $n$ vertices by going over all $\binom{n}{2}$ pairs of vertices, deciding independently with probability $p$ for a pair to be an edge. The model $\GNP$ is thus a probability space of all labeled graphs on the vertex set $[n]$ where the probability of such a graph, $G=([n],E)$, to be selected is $p^{|E|}(1-p)^{\binom{n}{2}-|E|}$. This product probability space provides us with a wide variety of probabilistic tools for analyzing the behavior of various random graph properties. (See monographs \cite{Bol2001} and \cite{JanLucRuc2000} for a thorough introduction to the subject of random graphs). In the subsequent sections we will need at some point to employ a slightly generalized model. Let $F\subseteq \binom{V}{2}$ be an arbitrary subset and let $\GNP_{-F} := \GNP \setminus F$.

Although the Binomial random graph model is very natural and relatively easy to use, it was not the first model to be considered. In their seminal paper, Erd\H{o}s and R\'{e}nyi considered the uniform probability space over all graphs on a fixed set of vertices with exactly $M$ edges, $\GNM$. Note that for any value of $p$, if we condition the random graph $\GNP$ to have exactly $M$ edges, then we obtain exactly the Erd\H{o}s-R\'{e}nyi random graph model. The similarity of the two models enables us to prove the occurrence of events in the $\GNP$ model and get the corresponding result in the $\GNM$ model.
\begin{prop}[\cite{JanLucRuc2000}, Proposition 1.13]\label{p:MonoGNPtoGNM}
Let $\mP = \mP(n)$ be a sequence of monotone increasing graph properties, $0\leq a\leq 1$ and $0\leq M\leq\binom{n}{2}$ be an integer. If for every sequence $p=p(n)\in[0,1]$ such that $p=M/\binom{n}{2}\pm O\left(M\left(\binom{n}{2}-M\right)/\binom{n}{2}^3\right)$ it holds that $\lim_{n\rightarrow\infty}\Prob{\GNP\in \mP}=a$, then $\lim_{n\rightarrow\infty}\Prob{\GNM\in \mP}=a$.
\end{prop}
The converse result to Proposition \ref{p:MonoGNPtoGNM} holds\footnote{In fact, when moving from $\GNM$ to $\GNP$ the monotonicity requirement is not necessary.} as well (see e.g. Proposition 1.12 in \cite{JanLucRuc2000}); this enables us to transfer results from one model to the other. Unfortunately, not all properties we will encounter and explore are monotone increasing, and hence Proposition \ref{p:MonoGNPtoGNM} cannot be used in those cases. Nonetheless, we would like to take advantage of the ``ease'' of calculations in the $\GNP$ model (due to the independence of appearance of its edges), and transfer the results to the $\GNM$ model, for the appropriate values of $M$. To achieve this we will use this somewhat crude estimate, which will suffice for our purposes.
\begin{clm}\label{c:GNPtoGNM}
Let $\mP$ be a property of graphs on $n$ vertices and let $1\leq M\leq\binom{n}{2}$ be an integer. Setting $p=M/\binom{n}{2}$ we have
$$\Prob{\GNM\in\mP}\leq\sqrt{2\pi M}\cdot\Prob{\GNP\in\mP}.$$
\end{clm}
\begin{proof}
Let $G \sim \GNM$ and $G' \sim \GNP$ where $1 \leq M \leq \binom{n}{2}$ is an integer and $p=M/\binom{n}{2}$. As was previously noted, we have $\cProb{}{G'\in\mP}{e(G')=M}=\Prob{G\in\mP}$. Next, we lower bound the probability the Binomial random graph will span exactly its expected number of edges using Stirling's formula. Let $N=\binom{n}{2}$, then
\begin{eqnarray*}
\Prob{e(\GNP)=M}&=&\binom{N}{M}p^M(1-p)^{N-M}>\frac{N^N\cdot p^M(1-p)^{N-M}}{M^M\cdot(N-M)^{N-M}}\cdot\frac{\sqrt{2\pi N}}{\sqrt{2\pi M}\cdot\sqrt{2\pi (N-M)}}>\frac{1}{\sqrt{2\pi M}}.
\end{eqnarray*}
Putting this together we have that
\begin{eqnarray*}
\Prob{G\in\mP}&=&\cProb{}{G'\in\mP}{e(G')=M}
=\frac{\Prob{G'\in\mP\wedge e(G')=M}}{\Prob{e(G')=M}}
\leq\frac{\Prob{G'\in\mP}}{\Prob{e(G')=M}}\\
&\leq&\sqrt{2\pi M}\cdot\Prob{G'\in\mP}.
\end{eqnarray*}
\end{proof}

Next, we consider the following generation process of graphs. Given
a set $V$ of $n$ vertices and an ordering on the pairs of vertices
$\pi:\binom{V}{2}\rightarrow\left[\binom{n}{2}\right]$, we define a
\emph{graph process} to be a sequence of graphs $\tG =
\tG(\pi)=\{G_t\}^{\binom{n}{2}}_{t=0}$ on $V$. Starting with
$G_0=(V,\emptyset)$, for every integer $1\leq t\leq \binom{n}{2}$,
the graph $G_t$ is defined by $G_t:=G_{t-1}\cup\pi^{-1}(t)$. For a
given graph process $\tG$ on $V$, we define the \emph{hitting time}
of a monotone increasing graph property $\mP$ on $V$ as
\begin{equation}\label{e:hitting_time}
\tau(\tG;\mP)=\min\{t\;:\;G_t\in \mP\}.
\end{equation}
When selecting $\pi$ uniformly at random, the process $\tG(\pi)$ is usually called the \emph{random graph process}. If $\tG=\{G_t\}_{t=0}^{\binom{n}{2}}$ is the random graph process, then, for every $0\leq M\leq\binom{n}{2}$, the graph $G_M$ is distributed according to $\GNM$, that is, $G_M\sim\GNM$. This entails that analyzing the hitting time of a monotone increasing property $\mP$ is in fact a refinement of the study of values of $M$ and $p$ for which $\GNM\in\mP$ and $\GNP\in\mP$ respectively (where to get the values of $p$ we need to use the converse of Proposition \ref{p:MonoGNPtoGNM} as stated above).

For every positive integer $k$ let $\delta_k$ denote the graph property of having minimum degree at least $k$, let $\EdgeConn{k}$ denote the graph property of being $k$-edge connected, let $\VertConn{k}$ denote the graph property of being $k$-vertex connected, and let $\Hamiltonicity$ denote the graph property of admitting a Hamilton cycle. Two cornerstone results in the theory of random graphs are that of Bollob\'{a}s and Thomason \cite{BolTho85} who proved that for every $1\leq k\leq n-1$, with high probability (or w.h.p. for brevity)\footnote{In this paper, we say that a sequence of events $\mA_n$ in a random graph model occurs w.h.p. if the probability of $\mA_n$ tends to 1 as the number of vertices $n$ tends to infinity.} $\tau(\tG;\delta_{k})=\tau(\tG;\EdgeConn{k})=\tau(\tG;\VertConn{k})$, and that of Koml\'os and Szemer\'edi \cite{KomSze83} who proved that w.h.p. $\tau(\tG;\delta_{2})=\tau(\tG;\Hamiltonicity)$ (see also \cite{Bol84}). Note that these two results (and many other which have succeeded) provide a very strong indication that the ``bottleneck'' for such properties in random graphs is in fact the vertices of minimum degree. The results of this paper are of the very same nature.

\subsection{Motivation and previous results}
Given a graph $G$ with minimum degree at most $2k-1$ Breaker can keep claiming edges incident to some vertex of minimum degree, and with the advantage of playing first will leave Maker with a graph containing a vertex of degree at most $k-1$. This implies that Breaker wins the $k$-edge-connectivity game $(E(G), \mF_{\EdgeConn{k}})$ for such graphs, and therefore $\tau(\tG;\mM_{\EdgeConn{k}})\geq\tau(\tG;\delta_{2k})$ for every graph process $\tG$. In \cite{StoSza2005} Stojakovi\'{c} and Szab\'{o} were the first to consider Maker-Breaker games played on random graphs. By combining theorems of Lehman \cite{Leh64} and of Palmer and Spencer \cite{PalSpe95}, they observed that for every fixed positive integer $k$, if $\tG$ is the random graph process, then w.h.p. $\tau(\tG;\mM_{\EdgeConn{k}})=\tau(\tG;\delta_{2k})$, thus providing a very precise hitting time result for the edge-connectivity game\footnote{In \cite{StoSza2005} only the case of $k=1$ is explicitly mentioned, but it can be generalized for any positive integer $k$ in a straightforward manner.}. Similarly to the edge-connectivity case we have that for every graph process $\tG$
\begin{equation}\label{e:lowbound_conn_game}
\tau(\tG;\delta_{2k})\leq\tau(\tG;\mM_{\VertConn{k}}).
\end{equation}
Let $\PerfMatch$ denote the graph property of admitting a matching of size $\lfloor n/2\rfloor $ in a graph on $n$ vertices. Every graph $G$ on an \emph{even} number of vertices with minimum degree at most $1$ is a win for Breaker in the perfect matching game $(E(G),\mF_{\PerfMatch})$. Hence, for every graph process $\tG$ on an even number of vertices
\begin{equation}\label{e:lowbound_perfmatch}
\tau(\tG;\delta_{2})\leq\tau(\tG;\mM_{\PerfMatch}).
\end{equation}

In \cite{StoSza2005} Stojakovi\'{c} and Szab\'{o} conjectured that if $\tG$ is the random graph process, then w.h.p. equality holds in \eqref{e:lowbound_perfmatch}. Although they did not prove this conjecture, in \cite{StoSza2005} they proved that if $p>\frac{64\ln n}{n}$, then w.h.p. $\GNP\in\mM_{\PerfMatch}$. Note that this result is optimal in $p$ up to multiplicative constant factor, for if $p\leq\frac{\ln n + \ln\ln n - \omega(1)}{n}$, where $\omega(1)$ is some function which tends to infinity with $n$ arbitrarily slowly, then w.h.p. $\delta(\GNP) \leq 1$, and hence by \eqref{e:lowbound_perfmatch}, w.h.p. $\GNP\notin\mM_\PerfMatch$.

Clearly, every graph $G$ with minimum degree at most $3$ is a win for Breaker in the Hamiltonicity game $(E(G), \mF_{\Hamiltonicity})$. Hence, we have that for every graph process $\tG$
\begin{equation}\label{e:lowbound_hamilton}
\tau(\tG;\delta_{4})\leq\tau(\tG;\mM_{\Hamiltonicity}).
\end{equation}
In \cite{StoSza2005} Stojakovi\'{c} and Szab\'{o} conjectured that if $\tG$ is the random graph process, then w.h.p. equality holds in \eqref{e:lowbound_hamilton}.

One of the first results in the field of Maker-Breaker games on graphs is due to Chv\'{a}tal and Erd\H{o}s in their seminal paper \cite{ChvErd78},  which states that $K_n\in \mM_\Hamiltonicity$ for sufficiently large values of $n$ (in~\cite{HefSti2009} the third author and Stich proved that $n\geq38$ suffices). The problem of finding sparse graphs which are a win for Maker was addressed by Hefetz et. al. \cite{HefEtAlPre} where they showed that, for sufficiently large values of $n$, there exists a graph $G\in \mM_\Hamiltonicity$ on $n$ vertices with $e(G)\leq 21n$. Playing the Hamiltonicity game $(E(G), \mF_{\Hamiltonicity})$ on the random graph $\GNP$ was first considered in the original paper of Stojakovi\'{c} and Szab\'{o} \cite{StoSza2005} where they proved that if $p>\frac{32\ln n}{\sqrt n}$, then w.h.p. $\GNP\in\mM_\Hamiltonicity$. Later, Stojakovi\'{c} \cite{Sto2005} found the correct order of magnitude proving that $p>5.4\ln n/n$ suffices for $\GNP$ to be w.h.p. Maker's win in the Hamiltonicity game. This requirement on $p$ was subsequently improved to $p\geq\frac{\ln n + (\ln\ln n)^s}{n}$, where $s$ is some large but fixed constant, by Hefetz et. al. \cite{HefEtAl2009}. Note that this result is very close to being optimal, for if $p=\frac{\ln n + 3\ln\ln n - \omega(1)}{n}$, where $\omega(1)$ is some function which tends to infinity with $n$ arbitrarily slowly, then w.h.p. $\delta(\GNP)<4$ and hence by \eqref{e:lowbound_hamilton} w.h.p. $\GNP\notin\mM_\Hamiltonicity$. Lastly, in \cite{BenKriSudPre} the first and fourth authors with Sudakov studied the Hamiltonicity game played on the edges of random regular graphs (the uniform probability measure over all $d$-regular graphs on a fixed vertex set) and proved that for large enough constant values of $d$ this game is Maker's win.

\subsection{Our results}
In this paper we address the above mentioned Maker-Breaker games on random graphs, namely when Maker's goal is to build graphs which satisfy the properties of being $k$-vertex connected, admitting a perfect matching, and being Hamiltonian. Specifically, the main objective of this paper is to prove that the trivial minimum degree requirement as stated in \eqref{e:lowbound_conn_game}, \eqref{e:lowbound_perfmatch}, and \eqref{e:lowbound_hamilton} is actually the bottleneck for a typical random graph to be a win for Maker in all of the above mentioned games. The following results will thus be proved.
\begin{our_thm}\label{t:kconn_game_hittime}
For every fixed integer $k\geq 1$, if $\tG$ is the random graph process, then w.h.p.
$$\tau(\tG;\mM_{\VertConn{k}})=\tau(\tG;\delta_{2k}).$$
\end{our_thm}
For every positive integer $k$ it holds that $\VertConn{k}\subseteq\EdgeConn{k}$, hence Theorem \ref{t:kconn_game_hittime} is in fact an improvement of the aforementioned result of Stojakovi\'{c} and Szab\'{o} in \cite{StoSza2005}. We also note that, by using the theorem of Lehman \cite{Leh64}, we can get the result of Palmer and Spencer \cite{PalSpe95} for even values of $k$ as a corollary of Theorem \ref{t:kconn_game_hittime}.

The following result for the prefect matching game is also proved.
\begin{our_thm}\label{t:perf_match_game_hittime}
If $\tG$ is the random graph process on an even number of vertices, then w.h.p.
$$\tau(\tG;\mM_{\PerfMatch})=\tau(\tG;\delta_{2}).$$
\end{our_thm}
Theorem \ref{t:perf_match_game_hittime} settles a conjecture raised in \cite{StoSza2005}. By the connection between the random graph models as described in Section \ref{ss:RandomGraphs} and by known results on the distribution of the minimum degree of $\GNP$, Theorem \ref{t:perf_match_game_hittime} implies that w.h.p. $\GNP\in\mM_\PerfMatch$ for every $p\geq\frac{\ln n + \ln\ln n+\omega(1)}{n}$, where $\omega(1)$ tends arbitrarily slowly to infinity with $n$, improving on the result of Stojakovi\'{c} and Szab\'{o} in \cite{StoSza2005}.
\begin{our_thm}\label{t:hamilton_game_hittime}
If $\tG$ is the random graph process, then w.h.p.
$$\tau(\tG;\mM_{\Hamiltonicity})=\tau(\tG;\delta_{4}).$$
\end{our_thm}
Theorem \ref{t:hamilton_game_hittime} settles a conjecture raised in \cite{StoSza2005}. Moreover, similarly to the above, Theorem \ref{t:hamilton_game_hittime} improves on the result of Hefetz et. al. in \cite{HefEtAl2009} by implying that w.h.p. $\GNP\in\mM_\Hamiltonicity$ for every $p\geq\frac{\ln n + 3\ln\ln n+\omega(1)}{n}$, where $\omega(1)$ tends arbitrarily slowly to infinity with $n$.

\subsection{Organization}
The rest of the paper is organized as follows. In Section \ref{s:Preliminaries} we provide some preliminary technical results about positional games, expanders, and random graphs, which will be needed in the course of our proofs. Section \ref{s:expander_game} is devoted to the analysis of a general game in which Maker's goal is to build an expander graph. This will give us a framework from which we can build on to prove the concrete results on the more natural games mentioned above. In Section \ref{s:Random} we prove some properties of random graphs and random graph processes that will be useful in the proofs of our main results. We then move on to provide the full proofs of Theorems \ref{t:kconn_game_hittime} and \ref{t:perf_match_game_hittime} in Section \ref{s:HitTimePerfMatchKConn}. These proofs will rely heavily on the general expander game and the properties of random graphs and random graph processes which we discussed in the preceding two sections. In Section \ref{s:HitTimeHam} we move on to the proof of Theorem \ref{t:hamilton_game_hittime}, which is more delicate than the previous two and requires some more ideas to get the result in full. Lastly, we discuss some further generalizations and sketch their proofs in Section \ref{s:Generalizations}.

\section{Preliminaries} \label{s:Preliminaries}
In this section we cite some tools which we will make use of in the succeeding sections. First,  we will need to employ bounds on large deviations of random variables. We will mostly use the following well-known bound on the lower and the upper tails of the Binomial distribution due to Chernoff (see e.g. \cite[Appendix A]{AloSpe2008}).
\begin{thm}[Chernoff bounds]\label{t:Chernoff}
If $X\sim B(n,p)$ then
\begin{enumerate}
\item $\Prob{X < (1-\varepsilon)np}<\exp(-\frac{\varepsilon^2np}{2})$ for every $\varepsilon>0$;
\item $\Prob{X >(1+\varepsilon)np}<\exp(-\frac{np}{3})$ for every $\varepsilon\geq 1$.
\end{enumerate}
\end{thm}
It will sometimes be more convenient to use the following bound on the upper tail of the Binomial distribution.
\begin{lem}\label{l:BeckChe}
If $X \sim Bin(n,p)$ and $k \geq np$, then $\Prob{X \geq k} \leq (enp/k)^k$.
\end{lem}
Note that the bound given in Lemma \ref{l:BeckChe} is especially useful when $k$ is ``much larger'' than $np$.

For the sake of simplicity and clarity of presentation, we do not make a particular effort to optimize the constants obtained in our proofs. We also omit floor and ceiling signs whenever these are not crucial. Most of our results are asymptotic in nature and whenever necessary we assume that $n$ is sufficiently large.

\subsection{Notation}
Our graph-theoretic notation is standard and follows that of  \cite{Wes2001}. In particular, we use the following. For a graph $G$, let $V(G)$ and $E(G)$ denote its sets of vertices and edges respectively, and let $e(G) = |E(G)|$. For a set $A \subseteq V(G)$, let $E_G(A)$ denote the set of edges of $G$ with both endpoints in $A$, and let $e_G(A) = |E_G(A)|$. For disjoint sets $A,B\subseteq V(G)$, let $E_G(A,B)$ denote the set of edges of $G$ with one endpoint in $A$ and the other in $B$, and let $e_G(A,B)=|E_G(A,B)|$. For a set $S \subseteq V(G)$, let $N_G(S) = \{u\in V(G)\setminus S:\exists v \in S, \{u,v\} \in E(G)\}$ denote the set of neighbors of $S$ in $V(G)\setminus S$. For a vertex $w \in V(G)$, we abbreviate $N_G(\{w\})$ to $N_G(w)$. For a vertex $w\in V(G) \setminus S$ let $d_G(w,S) = |\{u \in S : \{ u,w\} \in E(G)\}|$ denote the number of vertices of $S$ that are adjacent to $w$ in $G$. We abbreviate $d_G(w,V \setminus \{w\})$ to $d_G(w)$ which denotes the degree of $w$ in $G$. The minimum vertex degree in $G$ is denoted by $\delta(G)$. For a set $S \subseteq V(G)$ let $G[S]$ denote the subgraph of $G$ with vertex set $S$ and edge set $E_G(S)$. Let $c(G)$ and $o(G)$ respectively denote the number of connected components and the number of connected components of odd cardinality in $G$. Lastly, we will denote by $\ell(G)$ the length of a longest path in $G$, where the length of a path is the number of its edges.

\subsection{Basic positional games results}
The following theorem is a  classical result of Erd\H{o}s and Selfridge \cite{ErdSel73} which provides a useful sufficient condition for Breaker's win in the $(X, {\mF})$ game.
\begin{thm}[Erd\H{o}s and Selfridge \cite{ErdSel73}]\label{t:ErdSel}
For any hypergraph $(X,{\mF})$, if
$$\sum_{A\in \mF} 2^{-|A|} < \frac{1}{2},$$
then Breaker, playing as the first or second player, has a winning strategy for the $(X,{\mF})$ game.
\end{thm}

The following simple lemma is useful when a player is trying to ensure expansion of small sets. A similar lemma appeared in \cite{HefEtAl2009}.
\begin{lem}\label{l:quarterdegree}
For every integer $k>0$, if $H$ is a graph on $n$ vertices with minimum degree $\delta(H)\geq 5k$, then $H\in\mM_{\delta_k}$. Moreover, Maker can win the minimum degree $k$ game on the edge set of $H$ in at most $kn$ moves.
\end{lem}
\begin{proof}
We define a new graph $H^*$, where $H^* = H$ if all the degrees in $H$ are even, and otherwise $H^*$ is the graph obtained from $H$ by adding a new vertex $v^*$ and connecting it to every vertex of odd degree in $H$. Since all degrees of $H^*$ are even, it admits an Eulerian orientation $\overrightarrow{H}^*$. For every $v\in V(H)$, let $E(v) = \{\{v,u\}\in E(H)\;:\;\overrightarrow{(v,u)}\in E(\overrightarrow{H}^*)\}$. Clearly, $|E(v)|\geq d_{H^*}(v)/2 \geq \lfloor d_H(v)/2\rfloor\geq\lfloor 5k/2\rfloor$ and the sets $\{E(v)\}_{v\in V(H)}$ are pairwise disjoint. In every round, if Breaker claims an edge of $E(v)$, then Maker responds by claiming an edge of $E(v)\setminus \{\{v,v^*\}\}$, unless he already has $k$ edges incident with $v$ in which case Maker proceeds by claiming an edge of $E(u)$, where $u$ is some vertex such that Maker did not yet claim $k$ of its incident edges (if no such vertex exists, then the game was already won by Maker). Note that since $\lfloor|E(v)|/2\rfloor \geq k$, Maker can always play according to this strategy, and is never forced to pick an edge incident with $v^*$. Hence, Maker claims only edges of the original graph $H$. Disregarding the orientation, after at most $kn$ moves, the graph spanned by Maker's edges has minimum degree at least $k$ as claimed.
\end{proof}
\subsection{$(R,c)$-expanders}
Let us first define the type of expanders we wish to study.
\begin{defn}\label{d:rcexpander}
For every $c>0$ and every positive integer $R$ we say that a graph $G=(V,E)$ is an $(R,c)$-\emph{expander} if every subset of vertices $U\subseteq V$ of cardinality $|U|\leq R$ satisfies $|N_G(U)|\geq c\cdot |U|$. We denote the graph property of being an $(R,c)$-expander by $\mX_{R,c}$.
\end{defn}
\begin{rem}\label{r:rcexpander_mono}
From the above definition it clearly follows that for every $c>0$ and every positive integer $R$ (both $c$ and $R$ can be functions of the number of vertices of the graph in question), the graph property $\mX_{R,c}$ is monotone increasing.
\end{rem}

Next, we consider some structural properties of $(R,c)$-expanders. The following two claims show that the removal or addition of subsets that satisfy certain properties result in graphs that are still expanders. These properties will allow us to slightly modify certain expanders without losing their expansion properties.
\begin{clm}\label{c:indsetremovalfromexpander}
If $G=(V,E)$ is an $(R,c)$-expander and $U\subseteq V$ is a subset of vertices such that no two vertices of $U$ have a common neighbor in $G$, then $G[V\setminus U]$ is an $(R,c-1)$-expander.
\end{clm}
\begin{proof}
Let $S\subseteq V\setminus U$ be a set of cardinality $|S|\leq R$. It follows by our assumption on $U$ that $|N_G(v) \cap U| \leq 1$ holds for every vertex $v \in S$. Hence $|N_{G[V\setminus U]}(S)|\geq |N_G(S)|-|S|\geq (c-1)|S|$.
\end{proof}
\begin{clm}\label{c:indsetaddtoexpander}
Let $G=(V,E)$ be a graph, let $c>0$, and let $R$ be a positive integer. Let $U \subseteq V$ be a subset of vertices such that $d_G(u)\geq (c-1)$ for every $u\in U$, and, moreover, there is no path of length at most $4$ in $G$ whose (possibly identical) endpoints lie in $U$. If $G[V\setminus U]$ is an $(R,c)$-expander, then $G$ is an $(R,c-1)$-expander.
\end{clm}
\begin{proof}
Let $V'=V\setminus U$ and let $H=G[V']$. Let $S\subseteq V$ be of cardinality $s\leq R$, and set $S_1=S\cap U$ and let $S_2=S\setminus S_1$ with respective cardinalities $s_1$ and $s_2=s-s_1$. Our assumption on $U$ imply it is independent, and furthermore, for every $U'\subseteq U$ we have that $|N_G(U')|\geq (c-1)|U'|$. It follows that $N_G(S_1)\subseteq V\setminus U$. Furthermore, $N_G(S_1)$ can contain at most one vertex from each set $\{\{t\}\cup N_H(t)\}_{t\in V'}$, and hence $|N_G(S_1)\cap (S_2\cup N_H(S_2))|\leq |S_2|$. It follows that $N_G(S)\supseteq N_G(S_1)\cup (N_H(S_2)\setminus (N_G(S_1)\cap(S_2\cup N_H(S_2))))$, and that $|N_G(S)|\geq (c-1)s_1 + (c \cdot s_2 - s_2) = (c-1)s_1 + (c-1)(s-s_1)=(c-1)s$ as claimed.
\end{proof}

Next, we describe some sufficient conditions for a graph $G=(V,E)$ to be an expander (with appropriate parameters). Define:
\renewcommand{\labelenumi}{\textbf{(M\arabic{enumi})}}
\begin{enumerate}
\item\label{i:rcexpander_suffcondition_smallsets} $e_G(U)\leq\frac{\delta(G)|U|}{2(c+1)}$ for every subset of vertices $U \subseteq V$ of cardinality $1 \leq |U| < (c+1)r$;
\item\label{i:rcexpander_suffcondition_bigsets} $e_G(U,W)>0$ for every pair of disjoint subsets of vertices $U,W\subseteq V$ of cardinality $|U| = |W| = r$.
\end{enumerate}
\begin{lem}\label{l:rcexpander_suffcondition}
For every $c>0$, if $G=(V,E)$ is a graph which satisfies properties \textbf{M\ref{i:rcexpander_suffcondition_smallsets}} and \textbf{M\ref{i:rcexpander_suffcondition_bigsets}} for some positive integer $r\leq\frac{|V|}{c+2}$, then $G$ is a $(\frac{|V|-r}{c+1},c)$-expander.
\end{lem}
\begin{proof}
Set $R = \frac{|V|-r}{c+1}$; note that $R \geq r$ holds by the assumption of the lemma. Assume for the sake of contradiction that there exists a set $S\subseteq V$ of cardinality $|S|\leq R$ for which $|N_G(S)| < c|S|$. Let $T = S \cup N_G(S)$, then $|T| < (c+1)|S|$. If $1\leq|S|\leq r$, then $|T| < (c+1)r$. Moreover, since all edges that have at least one endpoint in $S$ are spanned by the vertices of $T$, it follows that $e_G(T) \geq \frac{\delta(G)|S|}{2}>\frac{\delta(G)|T|}{2(c+1)}$, which contradicts property \textbf{M\ref{i:rcexpander_suffcondition_smallsets}}. If $r<|S|\leq R$, then, since $e_G(S,V\setminus T)=0$ and $|V\setminus T|> |V|-(c+1)|S| \geq |V|-(c+1)R=r$, we obtain a contradiction to property \textbf{M\ref{i:rcexpander_suffcondition_bigsets}}. This concludes the proof of the lemma.
\end{proof}

The reason we study $(R,c)$-expanders is the fact that they entail some pseudo-random properties from which (under some conditions on $R$ and $c$) some of the natural properties that are considered in this paper, namely, admitting a perfect matching, being $k$-vertex-connected. We will provide a sufficient conditions for an $(R,c)$-expander to be $k$-vertex connected and to admit a perfect matching. Hence by playing for an $(R,c)$-expander, Maker will be able to win the two games whose goals are the aforementioned two properties (each posing different conditions on $R$ and $c$). The sufficient condition for a graph to be Hamiltonian, that we will use in the course of the proof, is more delicate than the conditions for $k$-vertex connectivity and for admitting a perfect matching, and requires some additional ideas, but the heart of the proof will still rely on expanders, and the same expander-game.

\section{An expander game on pseudo-random graphs}\label{s:expander_game}
The main object of this section is to describe a general Maker-Breaker game which will reside in the core of all of our proofs. Specifically, the goal of this section is to provide sufficient conditions for $G\in \mM_{\mX_{R,c}}$, or namely, for a graph $G$ to be Maker's win when Maker's goal is to build an $(R,c)$-expander. Although this game may seem at first to be an unnatural and artificial game to study, it turns out that this game will lie in the heart of our proofs of all of the results presented in this paper. Given parameters $c>0$, $0<\varepsilon<1$, $K>0$ and a positive integer $r\leq \frac{|V|}{c+1}$, we define the following two properties of a graph $H=(V,E)$ on $n'$ vertices. These properties, which are closely related to properties \textbf{M\ref{i:rcexpander_suffcondition_smallsets}} and \textbf{M\ref{i:rcexpander_suffcondition_bigsets}}, will be needed in the proof of the main result of this section. Define:
\renewcommand{\labelenumi}{\textbf{(Q\arabic{enumi})}}
\begin{enumerate}
\item\label{i:split_smallsets} $e_H(U)\leq\frac{\varepsilon\delta(H)|U|}{10(c+1)}$ for every subset of vertices $U \subseteq V$ of cardinality $1\leq|U|<(c+1)r$;
\item\label{i:split_bigsets} $e_H(U,W)\geq Kr\ln\left(\frac{n'}{r}\right)$ for every pair of disjoint subsets of vertices $U,W\subseteq V$ of cardinality $|U|=|W|=r$.
\end{enumerate}
\begin{rem}
Whenever we will cite property \textbf{Q\ref{i:split_bigsets}} we will give an explicit expression for $K$ which will not necessarily be a constant.
\end{rem}
\begin{thm}\label{t:expgame}
There exists an integer $n_0>0$ such that for every graph $G'=(V,E)$ on $n'\geq n_0$ vertices with minimum degree $\delta(G')>0$ and for every choice of parameters $\frac{1}{2\delta(G')}<\varepsilon<\frac{1}{2}$, $c>0$, and integer $0 < r \leq \min\{\frac{n'}{c+2},\frac{n'}{e^{30}}\}$ for which $G'$ satisfies properties \textbf{Q\ref{i:split_smallsets}} and \textbf{Q\ref{i:split_bigsets}} with $K=\frac{n'}{r(1-2\varepsilon)}$, Maker can win the $(\frac{n'-r}{c+1},c)$-expander game on $G'$, that is, $G'\in \mM_{\mX_{R,c}}$ with $R=\frac{n'-r}{c+1}$.
\end{thm}

Our proof of this theorem will be presented as a series of three lemmata whose composition implies Theorem \ref{t:expgame} directly.
\begin{lem}\label{l:split}
There exists an integer $n_0>0$ such that for every graph $G'=(V,E)$ on $n'\geq n_0$ vertices with minimum degree $\delta(G')>0$ and for every choice of parameters $\frac{1}{2\delta(G')}<\varepsilon<\frac{1}{2}$ and integer $0 < r \leq n'/e^4$ for which $G'$ satisfies property \textbf{Q\ref{i:split_bigsets}} with $K=\frac{n'}{r(1-2\varepsilon)}$, the edge set $E$ can be split into two disjoint subsets $E=E_1\cup E_2$ such that the graph $G_1=(V,E_1)$ has minimum degree $\delta(G_1)\geq\varepsilon\delta(G')$ and the graph $G_2=(V,E_2)$ satisfies property \textbf{Q\ref{i:split_bigsets}} with $K=3$.
\end{lem}
\begin{proof}
Pick every edge of $G'$ to be an edge in $G_1$ with probability $2\varepsilon$ independently of all other choices. The degree in $G_1$ of every vertex $v \in V$ is binomially distributed, that is, $d_{G_1}(v)\sim\Bin(d_{G'}(v),2\varepsilon)$ and thus its median is at least $\lfloor2\varepsilon\delta(G')\rfloor$. By our choice of $\varepsilon$ we have that $\lfloor2\varepsilon\delta(G')\rfloor > \varepsilon\delta(G')$ and therefore $\Prob{d_{G_1}(v)\geq\varepsilon\delta(G')}>1/2$. Since the degrees of every two vertices are positively correlated, we have that
$$\Prob{\delta(G_1)\geq\varepsilon\delta(G')}>2^{-n'}.$$
Let $U,W$ be a pair of disjoint subsets of vertices of cardinality $|U|=|W|=r$. By our assumption on $G'$ we have that $e_{G'}(U,W)\geq \frac{n'\ln\left(\frac{n'}{r}\right)}{1-2\varepsilon}$. As $e_{G_2}(U,W)\sim\Bin(e_{G'}(U,W),1-2\varepsilon)$ we have $\Exp{e_{G_2}(U,W)}\geq n'\ln\left(\frac{n'}{r}\right)$. Applying Theorem \ref{t:Chernoff} we have
$$\Prob{e_{G_2}(U,W)<3r\ln\left(\frac{n'}{r}\right)}\leq\exp\left(-\frac{\left(1-\frac{3r}{n'}\right)^2n'\ln\left(\frac{n'}{r}\right)}{2}\right)\leq\exp\left(-\frac{n'\ln\left(\frac{n'}{r}\right)}{3}\right).$$
By applying the union bound over all pairs of disjoint subsets of vertices of cardinality $r$ each, we conclude that the probability that $G_2$ violates property \textbf{Q\ref{i:split_bigsets}} with $K=3$ is at most
\begin{eqnarray*}
\binom{n'}{r}\binom{n'-r}{r}\exp\left(-\frac{n'\ln\left(\frac{n'}{r}\right)}{3}\right)&\leq&\left(\frac{en'}{r}\right)^{2r}\cdot\exp\left(-\frac{n'\ln\left(\frac{n'}{r}\right)}{3}\right)\\
&=&\exp\left(2r\left(1+\ln\left(\frac{n'}{r}\right)\right)-\frac{n'\ln\left(\frac{n'}{r}\right)}{3}\right)\\
&\leq&\exp\left(-\frac{n'\ln\left(\frac{n'}{r}\right)}{4}\right)\\
&<& 2^{-n'},
\end{eqnarray*}
and therefore there exists a partition of $G'$ as claimed.
\end{proof}

The following lemma provides a sufficient condition on a graph $G=(V,E)$ for it to be a Maker's win in the game $(E, \mathcal{F}_{M\ref{i:rcexpander_suffcondition_bigsets}})$, that is, the game on $G$ in which Maker's goal is to build a subgraph which satisfies the (monotone increasing) property \textbf{M\ref{i:rcexpander_suffcondition_bigsets}}. In order to prove this result, we invoke a rather standard technique of studying a dual game in which the roles of Maker and Breaker are exchanged. Note that in the dual game, Breaker (which was the original Maker) is the second player.
\begin{lem}\label{l:large_sets_game}
There exists an integer $n_0>0$ such that for every graph $G_2=(V,E_2)$ on $n'\geq n_0$ vertices and for every integer $0 < r\leq n'/e^{30}$ for which $G_2$ satisfies property \textbf{Q\ref{i:split_bigsets}} with $K=3$, playing on $E_2$ Maker can build a subgraph of $G_2$ which satisfies property \textbf{M\ref{i:rcexpander_suffcondition_bigsets}}.
\end{lem}
\begin{proof}
Let $G_2$ be any graph with vertex set $V$. In order for Maker to build a graph which satisfies property \textbf{M\ref{i:rcexpander_suffcondition_bigsets}}, he can adopt the role of Breaker in the game $(E_2,\mL)$, where $\mL$ is the family of edge-sets of all induced bipartite subgraphs of $G_2$ with both parts of size $r$. Recall that, by property \textbf{Q\ref{i:split_bigsets}} with $K=3$, every such winning set $L\in \mL$ spans at least $3r\ln\left(\frac{n'}{r}\right)$ edges. It follows that
\begin{eqnarray*}
\sum_{L \in \mL}2^{-|L|}&\leq&\sum_{\substack{U\subseteq V\\|U|=r}}\sum_{\substack{W\subseteq V\setminus U\\|W|=r}} 2^{- e_{G_2}(U,W)}\\
&\leq&\binom{n'}{r}\binom{n'-r}{r}\cdot\exp\left(-3r\ln\left(\frac{n'}{r}\right)\ln 2\right)\\
&\leq& \left(\frac{e n'}{r}\right)^{2r} \cdot \exp\left(-3r\ln\left(\frac{n'}{r}\right)\ln 2\right)\\
&\leq&\exp\left(r\cdot\left(2\ln \left(\frac{n'}{r}\right)+2-\ln2\cdot3\ln\left(\frac{n'}{r}\right)\right)\right)\\
&<&\frac{1}{2}.
\end{eqnarray*}
The assertion of the lemma follows readily by Theorem~\ref{t:ErdSel}.
\end{proof}

\begin{lem} \label{l:expanderGame}
There exists an integer $n_0>0$ such that for every graph $G'=(V,E)$ on $n'\geq n_0$ vertices and for every choice of parameters $0<\varepsilon<1$, $c>0$ and integer $0<r\leq \frac{n'}{c+2}$ for which $G'$ satisfies property \textbf{Q\ref{i:split_smallsets}} and whose edge set can be partitioned into two disjoint sets $E=E_1\cup E_2$ where $G_1=(V,E_1)$ is of minimum degree $\delta(G_1)\geq\varepsilon\cdot\delta(G')$, and $G_2=(V,E_2)$ satisfies \textbf{Q\ref{i:split_bigsets}} with $K=3$, Maker can win the $(\frac{n'-r}{c+1}, c)$-expander game, that is, $G'\in \mM_{\mX_{R,c}}$ with $R=\frac{n'-r}{c+1}$.
\end{lem}
\begin{proof}
Before the game starts, Maker splits the board into two parts, $G_1=(V,E_1)$ and $G_2=(V,E_2)$ as indicated in the lemma. Maker then plays two separate games in parallel, one on $E_1$ and the other on $E_2$. In every turn in which Breaker claims some edge of $E_i$, for $i=1,2$, Maker responds by claiming an edge of $E_i$ as well (except for maybe once if Breaker has claimed the last edge of $E_i$). Let $H$ denote the graph built by Maker by the end of the game and set $H_1=(V, E(H)\cap E_1)$ and $H_2 =(V,E(H)\cap E_2)$.

The game on $E_1$ is played according to Lemma \ref{l:quarterdegree}. Hence, at the end of the game, Maker's graph $H_1$ will have minimum degree at least $\delta(H_1)\geq \frac{\delta(G_1)}{5}$. Since $G'$ satisfies property \textbf{Q\ref{i:split_smallsets}} and $\delta(G_1)\geq\varepsilon\delta(G')$ it follows that, for every $U\subseteq V$ of cardinality $1\leq|U|<(c+1)r$, the number of Maker's edges with both endpoints in $U$ is $e_H(U)\leq e_{G'}(U)\leq\frac{\varepsilon\delta(G')|U|}{10(c+1)}\leq\frac{\delta(G_1)|U|}{10(c+1)}\leq\frac{\delta(H_1)|U|}{2(c+1)}\leq\frac{\delta(H)|U|}{2(c+1)}$. Hence, $H$ satisfies property \textbf{M\ref{i:rcexpander_suffcondition_smallsets}}.

The game on $E_2$ is played according to Lemma \ref{l:large_sets_game}, and therefore at the end of the game, Maker will build a graph $H_2$ which satisfies property \textbf{M\ref{i:rcexpander_suffcondition_bigsets}}. By the monotonicity of \textbf{M\ref{i:rcexpander_suffcondition_bigsets}}, this property also holds for $H$. Noting that $H$, $n'$, $r$ and $c$ satisfy the conditions of Lemma \ref{l:rcexpander_suffcondition}, we deduce that $H\in\mM_{\mX_{R,c}}$, that is, Maker's graph is an $(R,c)$-expander as claimed.
\end{proof}

\section{Properties of random graphs and random graph processes}
\label{s:Random}
We start with a very simple claim regarding the number of edges in the Binomial random graph model $\GNP$.
\begin{clm}\label{c:MaxEdgesGNP}
If $p\geq\frac{\ln n}{n}$ and $G \sim \GNP$, then w.h.p. $e(G)\leq n^2p$.
\end{clm}
\begin{proof}
This is a simple application of Theorem \ref{t:Chernoff}. Clearly, $e(G)\sim\Bin(\binom{n}{2},p)$, entailing $\Prob{e(G)>n^2p}<\Prob{e(G)\geq2\binom{n}{2}\cdot p}\leq\exp(-\frac{(n-1)\ln n}{6})=o(1).$
\end{proof}
Next, we consider the random graph model we are interested in, the random graph process. For every fixed integer $k\geq 1$ we define two functions as follows:
\begin{eqnarray}
m_k&=&\binom{n}{2}\frac{\ln n + (k-1)\ln\ln n - \ln\ln\ln n}{n};\\
M_k&=&\binom{n}{2}\frac{\ln n + (k-1)\ln\ln n + \ln\ln\ln n}{n}.
\end{eqnarray}
The following lemma (see e.g. \cite{Bol2001}) describes a fairly precise behavior of the minimum degree of the random graph process.
\begin{lem}\label{l:hittime_mindeg}
For every fixed integer $k\geq 1$, if $\tG$ is the random graph process, then w.h.p.
$$m_k<\tau(\tG;\delta_k)<M_k.$$
\end{lem}

Let $G=(V,E)$ be a graph on $n$ vertices and, for a positive integer $t$, let
\begin{equation}
\mD_t=\mD_t(G)=\{v\in V\;:\; d_G(v)<t\}.
\end{equation}
\begin{rem}\label{r:Small-monotone}
Let $\tG=\{G_i\}_{i=0}^{\binom{n}{2}}$ be the random graph process, then $\mD_t(G_{i-1})\supseteq \mD_t(G_i)$ holds for every $1\leq i\leq \binom{n}{2}$.
\end{rem}
The following estimate on the probability of a vertex to be in $\mD_t(\GNP)$ will be of use later on.
\begin{clm}\label{c:prob_v_in_SMALL}
For every integer $t\leq\ln^{0.9} n$ and for every vertex $v$, if $\frac{\ln n}{n}<p<\frac{2\ln n}{n}$, then
$$\Prob{v\in\mD_t(\GNP)}\leq n^{-1+o(1)}.$$
\end{clm}
\begin{proof}
Let $G=(V,E)\sim\GNP$, then $d_G(v)\sim\Bin(n-1,p)$ holds for every vertex $v\in V$. It follows that $\Prob{v\in\mD_t(G)} = \Prob{\Bin(n-1,p)<t}$ holds for every $t$. Hence, for every integer $t\leq \ln^{0.9}n$ we have
\begin{eqnarray*}
\Prob{v\in\mD_t(G)}&=&\Prob{\Bin(n-1,p)<t}\\
&\leq&\sum_{i=0}^{t-1} \binom{n-1}{i}p^{i}(1-p)^{n-1-i}\\
&\leq&t\binom{n}{t}p^{t}(1-p)^{n-1-t}\\
&\leq&t\cdot\left(\frac{enp}{t}\right)^{t}e^{-p(n-1-t)}\\
&\leq&\ln^{0.9}n\cdot n^{o(1)}\cdot n^{-1+o(1)}\\
&\leq& n^{-1+o(1)}.
\end{eqnarray*}
\end{proof}

Next, we prove and cite some structural properties of the set $\mD_t(\GNM) = \mD_t(G_M)$. In order to prove these results, we resort to the use of $\GNP$, where the analysis is much simpler, and then use Claim \ref{c:GNPtoGNM} to transfer the results to the random graph model $\GNM$.
\begin{clm}\label{c:Small-size}
For every integer $t\leq\ln^{0.9}n$, if $\tG=\{G_i\}_{i=0}^{\binom{n}{2}}$ is the random graph process and $M\geq\tau(\tG;\delta_1)$, then w.h.p. $|\mD_t(G_{M})|\leq n^{0.3}$.
\end{clm}
\begin{proof}
By Lemma~\ref{l:hittime_mindeg} and Remark \ref{r:Small-monotone} it suffices to prove the claim for $M=m_1$. Set $p=M/\binom{n}{2}>0.9\ln n/n$ and let $G\sim\GNP$. Fix a subset $U$ of cardinality $|U|=\lfloor n^{0.3}\rfloor$. We bound the probability that all vertices $U$ have less than $t$ edges emitting outside of $U$. Denote by $N=|V\setminus U|=(1-o(1))n$, and let $u\in U$ be some vertex, then $e_G(u,V\setminus U)\sim\Bin(|V\setminus U|,p)$, and therefore
\begin{eqnarray*}
\Prob{e_G(u,V\setminus U)<t}&\leq&\sum_{i=0}^t\binom{n}{i}p^i(1-p)^{n-i}\\
&\leq&\sum_{i=0}^t\exp\left(i\cdot (1+\ln np -\ln i) - p(n-i)\right)\\
&\leq&n^{-0.89}
\end{eqnarray*}
As the number of edges emitting out of $U$ from each vertex in $U$ are independent random variables (each counting the appearance of edges from a disjoint set of the others), we have that the probability that from every vertex in $U$ there are less than $t$ edges emitting out of $U$ is at most $n^{-0.89|U|}$. There are $\binom{n}{|U|}$ subsets of this cardinality, hence applying the union bound over all these sets yields that the probability there exists such a set $U$ is at most $$\binom{n}{|U|}\cdot n^{-0.89|U|}\leq\exp\left(|U|\cdot(1+\ln\frac{n}{|U|}-0.89\ln n)\right)\leq\exp\left(-0.18|U|\ln n)\right)\leq e^{-n^{0.3}}.$$
By the definition all the vertices of of $\mD_t(G)$ have less than $t$ edges emitting out of it, hence the probability that $|\mD_t(G)|>n^{0.3}$ is at most $e^{-n^{0.3}}$. Applying Claim \ref{c:GNPtoGNM} we have that $\Prob{|\mD_t(G_M)|>n^{0.3}}\leq\sqrt{2\pi M}\cdot\exp\left(-n^{0.3}\right)<2\sqrt{n\ln n}\cdot\exp\left(-n^{0.3}\right)=o(1)$. This concludes the proof of the claim.
\end{proof}

\begin{clm}\label{c:Small-properties}
For every fixed integer $k\geq 1$ and for every integer $t\leq\ln^{0.9} n$, if $\tG=\{G_i\}_{i=0}^{\binom{n}{2}}$ is the random graph process and $M=\tau(\tG;\delta_k)$, then w.h.p $G=G_M$ does not contain a non-empty path of length at most $4$ such that both of its (possibly identical) endpoints lie in $\mD_t(G_M)$.
\end{clm}
\begin{proof}
Clearly, it suffices to consider the case $t=\ln^{0.9}n$. We will
prove the claim for two distinct endpoints in $\mD_t(G_M)$, and for
paths of length $2\leq r\leq 4$ between them, where the other cases
are similar (and a little simpler). By Lemma \ref{l:hittime_mindeg}
we can assume that $m_k<\tau(\tG;\delta_k)<M_k$, and hence it
follows by Remark \ref{r:Small-monotone} that
$\mD_t(G_M)\subseteq\mD_t(G_{m_k})$. Furthermore, if there exists a
path of length $r$ connecting two vertices of $\mD_t(G_M)$, then all
edges of this path are present in the graph $G_{M_k}$ as well.
Combining these two observations, we can upper bound the probability
that a path of length $r$ in $G_M$ connects two vertices of
$\mD_t(G_M)$ by the probability that a path of length $r$ in
$G_{M_k}$ connects two vertices of $\mD_t(G_{m_k})$. Thus, in fact
we need to analyze the random graph process $\tG$ at two different
points, $G_{m_k}$ and $G_{M_k}$. We do this by considering the
following setting where $G\sim\GNp{M_k}$ and $H\subseteq G$ is a
random subgraph of $G$, generated by selecting uniformly at random
$m_k$ of the edges of $G$.

Let $u,w \in V(G)$ and let $P = (u=v_0,\ldots,v_r=w)$ be a sequence of vertices of $V(G)$, where $2\leq r\leq 4$. Denote by $\mA_P$ the event $\{v_i, v_{i+1}\} \in E(G)$ for every $0 \leq i \leq r-1$. Denoting $N=\binom{n}{2}$ we have
\begin{equation}\label{e:Prob_P_Small}
\Prob{\mA_P}=\left(\frac{\binom{N-r}{M_k-r}}{\binom{N}{M_k}}\right)<1.01\left(\frac{\ln n}{n}\right)^r.
\end{equation}

Let $t'=\frac{\ln n}{10}$ and let $u,w \in V(G)$. Denote by $\mB_{u,w}$ the event that both $u$ and $w$ are elements of $\mD_t(H)$, by $\mB'_{u,w}$ the event that both $u$ and $w$ are elements of $\mD_{t'}(G)$, and by $\mC_{u,w}$ the event that, for each of $u$ and $w$, at most $t$ incident edges of $G$ were selected to be in $H$. Denote by $X_{u,w} = X_{u,w}(G)$ the random variable which counts the number of edges of $G$ that are incident with $u$ or $w$. Denote by $\mB''_{u,w}$ the event $X_{u,w}\leq 2t'$, then clearly $\mB'_{u,w} \subseteq \mB''_{u,w}$. Clearly, if $\{u,w\} \subseteq \mD_t(H)$, then either $\{u,w\} \subseteq \mD_{t'}(G)$ or $\{u,w\}$ is not contained in $\mD_{t'}(G)$ but, for each of $u$ and $w$, at most $t$ incident edges of $G$ were selected to be in $H$. It follows that $\mB_{u,w}\subseteq\mB''_{u,w}\cup\left(\mC_{u,w}\cap \overline{\mB''_{u,w}}\right)$. Putting this together implies
\begin{equation*}
\Prob{\mB_{u,w}\wedge\mA_P}\leq\Prob{\mA_P}\cdot\left(\cProb{}{\mB''_{u,w}}{\mA_P}+\cProb{}{\mC_{u,w}}{\mA_P\wedge\overline{\mB''_{u,w}}}\right).
\end{equation*}

When considering $X_{u,w}(G)$, the conditioning on $\mA_P$ and the fact that $r\geq 2$ implies that the two edges $\{u,v_1\}$ and $\{v_{r-1},w\}$ are present in $G$. It follows that $\left[(X_{u,w}-2)\right|\left.\mA_P\right]$ is distributed according the to hypergeometric distribution with parameters $N-r$, $M_k-r$, and  $2n-5$. Hence
\begin{eqnarray*}
\cProb{}{\mB''_{u,w}}{\mA_P}&=&\cProb{}{X_{u,w}-2 \leq 2t'-2}{\mA_P}\\
&\leq& \sum_{j=0}^{2t'-2}\binom{2n-5}{j}\cdot\frac{\binom{N-r-2n+5}{M_k-r-j}}{\binom{N-r}{M_k-r}}\\
&\leq&2t'\cdot\binom{2n}{2t'}\cdot\left(\frac{M_k-r}{N-M_k+2t'}\right)^{2t'}\cdot\left(\frac{N-r-2n+5}{N-r}\right)^{M_k-r-2t'}\\
&\leq&2t'\cdot\left(\frac{en(M_k-r)}{t'(N-M_k+2t')}\right)^{2t'}\cdot\exp\left(-(M_k-r-2t')\cdot\frac{2n-5}{N-r}\right)\\
&\leq&2t'\cdot\exp\left(2t'\ln30-1.9\ln n\right)\\
&\leq&2t'\cdot\exp\left(-1.21\ln n\right)\\
&\leq&0.45n^{-1.2}.\\
\end{eqnarray*}
Next, we bound $\cProb{}{\mC_{u,w}}{\mA_P\wedge\overline{\mB''_{u,w}}}$. Let $Y_{u,w} = Y_{u,w}(G)$ denote the number of edges of $H$ that are incident with $u$ or $w$, disregarding the edges of $P$ and the edge $\{u,w\}$ if it is in $G$. Hence, we can upper bound $\cProb{}{\mC_{u,w}}{\mA_P\wedge\overline{\mB''_{u,w}}}$ by
\begin{eqnarray*}
\cProb{}{Y_{u,w}\leq 2t-2}{X_{u,w}>2t'}&\leq&\sum_{j=0}^{2t-2}\binom{2t'}{j}\cdot\frac{\binom{M_k-2t'}{m_k-j}}{\binom{M_k}{m_k}}\\
&\leq&2t\cdot\binom{2t'}{2t}\cdot\left(\frac{m_k}{M_k}\right)^{2t}\cdot\left(\frac{M_k-m_k}{M_k-2t}\right)^{2t'-2t}\\
&\leq&2t\cdot\left(\frac{et'm_k}{tM_k}\right)^{2t}\cdot\exp\left(-(2t'-2t)\cdot\ln\left(\frac{M_k-2t}{M_k-m_k}\right)\right)\\
&\leq&\exp\left(\ln (2t)+0.2t\ln\ln n-2t'\cdot\omega(1)\right)\\
&\leq&0.45n^{-1.2}.
\end{eqnarray*}
Plugging in our upper bound for $\Prob{\mA_P}$ from \eqref{e:Prob_P_Small}, we conclude that for every choice of a sequence $P = (u=v_0,\ldots,v_r=w)$ of vertices of $V(G)$, where $2\leq r\leq 4$, the probability that $u,w\in\mD_t(G_M)$ and $\{v_i,v_{i+1}\} \in E(G_M)$ for every $0 \leq i \leq r-1$, is at most $1.01\left(\frac{\ln n}{n}\right)^r\cdot(0.45n^{-1.2}+0.45n^{-1.2})<\frac{\ln^rn}{n^{r+1.2}}$. The number of such sequences of length $r$ is at most $(r+1)!\binom{n}{r+1} \leq n^{r+1}$. Hence, applying a simple  union bound argument over all such sequences we conclude that the probability there exists a path in $G_M$ of length $r \leq 4$, connecting two vertices of $\mD_t(G_M)$, is at most  $n^{r+1} \cdot \frac{\ln^r n}{n^{r+1.2}}=o(1)$, as claimed.
\end{proof}

\begin{clm}\label{c:GNM_smallsets}
For every fixed integer $k\geq 2$, if $\tG=\{G_i\}_{i=0}^{\binom{n}{2}}$ is the random graph process and $M=\tau(\tG;\delta_k)$, then w.h.p. $G_M = (V,E)$ is such that $e_{G_M}(U)<|U|\ln^{0.8}n$ for every subset of vertices $U\subseteq V$ of cardinality $1\leq|U|\leq\frac{n}{\ln^{0.3}n}$.
\end{clm}
\begin{proof}
By Lemma \ref{l:hittime_mindeg} we can assume that $M<M_k$. As the complement of the property at hand is monotone increasing, it follows by Proposition \ref{p:MonoGNPtoGNM} that it suffices to prove that, if $p = p(n) \leq 2\ln n/n$ and $G\sim\GNP$, then the probability that there exists a subset $U\subseteq V$ of cardinality $1\leq|U|\leq\frac{n}{\ln^{0.3}n}$ such that $e_G(U) \geq |U|\ln^{0.8}n$, tends to $0$ as $n$ tends to infinity. Fix a subset $U$ of cardinality $1\leq u\leq n\cdot\ln^{-0.3}n$, then $e_G(U)\sim Bin(\binom{u}{2},p)$. Since $u\cdot\ln^{0.8}n\geq\binom{u}{2}\cdot p$, we can apply Lemma \ref{l:BeckChe} to upper bound the probability that $e_G(U)$ is too large. We can then upper bound the probability the claim is violated by applying a union bound argument as follows
\begin{eqnarray*}
\sum_{u=1}^{n\cdot\ln^{-0.3}n}\binom{n}{u}\Prob{e_G(U)\geq u\cdot\ln^{0.8}n}&\leq& \sum_{u=1}^{n\cdot\ln^{-0.3}n}\left(\frac{en}{u}\right)^{u}\cdot\left(\frac{e\binom{u}{2}p}{u\cdot\ln^{0.8}n}\right)^{u\cdot\ln^{0.8}n}\\
&\leq&\sum_{u=1}^{n\cdot\ln^{-0.3}n}\left(e^{\ln^{0.8} n +1}\cdot\left(\frac{u}{n}\right)^{\ln^{0.8}n -1}\cdot(\ln^{0.2}n)^{\ln^{0.8} n}\right)^u\\
&\leq&\sum_{u=1}^{n\cdot\ln^{-0.3}n}\left(4\cdot\left(\frac{u}{n}\right)^{0.99}\cdot(\ln^{0.2}n)\right)^{u\ln^{0.8} n}\\
&\leq&\sum_{u=1}^{n\cdot\ln^{-0.3}n}\left(\ln^{-0.09}n\right)^{u\ln^{0.8} n}\\
&=&o(1),
\end{eqnarray*}
where the last equality follows from the fact that we are summing a geometric series with a first element and quotient both being $o(1)$. This concludes the proof of the claim.
\end{proof}
\begin{clm}\label{c:GNM_bigsets}
For every fixed integer $k\geq 1$ and for an integer $r = \frac{n}{2\ln^{0.4}n}$, if $\tG=\{G_i\}_{i=0}^{\binom{n}{2}}$ is the random graph process and $M=\tau(\tG;\delta_k)$, then w.h.p. $e_{G_M}(U,W)\geq n\ln^{0.1}n$ for every pair of disjoint subsets $U,W\subseteq V(G_M)$ of cardinality $|U|=|W|=r$.
\end{clm}
\begin{proof}
By Lemma \ref{l:hittime_mindeg} we can assume that $M>m_k$. As the property at hand is monotone increasing, it follows by Proposition \ref{p:MonoGNPtoGNM} that it suffices to prove the claim for $G\sim\GNP$ with $p\geq\frac{\ln n}{n}$. Fix a pair of disjoint subsets $U, W \subseteq V(G)$ of cardinality $r$ each. Then $e_G(U,W)\sim\Bin(r^2,p)$, and thus $\Exp{e_G(U,W)}\geq\frac{n\ln^{0.2}n}{4}$. We upper bound the probability that $e_G(U,W)$ is too large using Theorem \ref{t:Chernoff}. We can then upper bound the probability the claim is violated by applying a union bound argument as follows
\begin{eqnarray*}
\binom{n}{r}\binom{n-r}{r}\Prob{e_G(U,W)< n\ln^{0.1}n}&\leq&\left(\frac{en}{r}\right)^{2r}\exp\left(-\frac{\left(1-\frac{4}{\ln^{0.1}n}\right)^2r^2p}{2}\right)\\
&\leq&\exp\left(r\left(2+\ln\ln n-\frac{\ln^{0.2}n}{10}\right)\right)\\
&=&o(1).
\end{eqnarray*}
This concludes the proof of the claim.
\end{proof}

Finally, we prove that removing vertices of small degree from a random graph with an appropriate number of edges typically results in a graph on which Maker can win the expander game. In fact, we even show that Maker can win the game when this graph is thinned substantially (that is, the vast majority of edges are removed). This stronger property will play a crucial role in the proof of Theorem \ref{t:hamilton_game_hittime}. Our proof will make use of results we have obtained in Claims \ref{c:Small-size}, \ref{c:GNM_smallsets}, \ref{c:GNM_bigsets} and in Lemma \ref{l:split}.
\begin{lem}\label{l:splitGNM_noSMALL}
For every $\alpha>0$ and for every fixed integer $k\geq2$, if $\tG=\{G_i\}_{i=0}^{\binom{n}{2}}$ is the random graph process and $M=\tau(\tG;\delta_k)$, then w.h.p. $G'=(V',E') := G_M \setminus \mD_{\ln^{0.9}n}(G_M)$ on $n'$ vertices contains a spanning subgraph $\hG\subseteq G'$ with at most $2n'\ln^{0.97}n'$ edges, such that $\hG\in\mM_{\mX_{R,c}}$ for every $0<c\leq\ln^{0.02}n'$ and $R\leq (1-\alpha)\frac{n'}{c+1}$.
\end{lem}
\begin{rem}
As was noted in Remark \ref{r:rcexpander_mono}, by the monotonicity of $\mX_{R,c}$, the above lemma can be used to deduce that $G'\in\mM_{\mX_{R,c}}$.
\end{rem}
\begin{proof}
Pick every edge of $G'$ to be an edge of $\hG$ with probability $\gamma=\ln^{-0.03}n$, independently of all other choices. Our goal is to prove that, with positive probability, $\hG$ satisfies the conditions of Theorem \ref{t:expgame}, with parameters
$$\varepsilon=\gamma \qquad \textrm{and} \qquad r=\frac{n'}{\ln^{0.4}n'}.$$
Based on typical properties of the random graph process, we can assume that $G'$ satisfies the following properties:
\begin{enumerate}
\renewcommand{\labelenumi}{\arabic{enumi})}
\item $\delta(G')\geq \ln^{0.9}n$;
\item $e(G')\leq e(G_M) \leq e(G_{M_k}) \leq (1+o(1)) \frac{n\ln n}{2}$ (Lemma \ref{l:hittime_mindeg});
\item $|{\mathcal D}_{\ln^{0.9}n}(G_M)|\leq n^{0.3}$, and therefore $n'\geq n(1-n^{-0.7})$ (Claim \ref{c:Small-size});
\item Every set $U\subseteq V'$ of cardinality $|U|\leq (c+1)r\leq\frac{n}{\ln^{0.3}{n}}$ satisfies $e_{G'}(U) = e_{G_M}(U) \leq |U|\ln^{0.8}n\leq|U|\ln^{0.81}n'$ (Claim \ref{c:GNM_smallsets});
\item Every pair of disjoint subsets $U,W\subseteq V'$ of cardinality $|U|=|W|=r \geq \frac{n}{2\ln^{0.4}n}$ satisfies $e_{G'}(U,W)\geq n\ln^{0.1}n$ (Claim \ref{c:GNM_bigsets}).
\end{enumerate}
It follows that our choice of parameters meets the requirements on $\varepsilon$, $c$ and $r$, made in Theorem \ref{t:expgame}.

We proceed to prove that, with a ``not too small'' probability, $\hG$ satisfies property \textbf{Q\ref{i:split_smallsets}}. First note that every set $U\subseteq V'$ of cardinality $|U|\leq(c+1)r$ satisfies $e_{\hG}(U)\leq e_{G'}(U)\leq|U|\ln^{0.81}n'$. The degree in $\hG$ of every vertex $v \in V'$ is binomially distributed, $d_{\hG}(v)\sim\Bin(d_{G'}(v),\gamma)$, with median at least $\lfloor\gamma\delta(G')\rfloor$. Therefore $\Prob{d_{\hG}(v)\geq\lfloor\gamma\delta(G')\rfloor} \geq 1/2$. Since $\delta(G')\geq\ln^{0.9}n$ and since the degrees of every two vertices are positively correlated, using the FKG inequality (see e.g. \cite[Chapter 6]{AloSpe2008}) we have that
$$\Prob{\delta(\hG)\geq\lfloor\ln^{0.87}n\rfloor} \geq \Prob{\delta(\hG)\geq\lfloor\gamma\delta(G')\rfloor} \geq 2^{-n'}.$$
It follows that with probability at least $2^{-n}$ we have $\frac{\varepsilon\delta(\hG)}{10(c+1)}>\ln^{0.81}n'$, and thus $\hG$ satisfies property \textbf{Q\ref{i:split_smallsets}} with probability at least $2^{-n'}$.

Next, we prove that, with ``very large'' probability, $\hG$ satisfies property \textbf{Q\ref{i:split_bigsets}}. Fixing a pair of disjoint sets of vertices $U,W\subseteq V'$ of cardinality $r$ each, it clearly follows that $e_{\hG}(U,W)\sim\Bin(e_{G'}(U,W),\gamma)$, and thus $\Exp{e_{\hG}(U,W)}\geq n\ln^{0.07}n>n'\ln^{0.07}n'$. Since $\frac{n'}{(1-2\varepsilon)}\ln\left(\frac{n'}{r}\right)\leq n'\ln^{0.05}n'$, we can upper bound the probability that the pair $U,W$ does not satisfy property \textbf{Q\ref{i:split_bigsets}} with $K=\frac{n'}{r(1-2\varepsilon)}$, using Theorem \ref{t:Chernoff}.
$$\Prob{e_{\hG}(U,W)<n'\ln^{0.05}n'}\leq\exp\left(-\frac{(1-\ln^{-0.02}n')^2n'\ln^{0.07}n'}{2}\right)\leq \exp\left(-\frac{n'\ln^{0.07}n'}{3}\right).$$
Applying a simple union bound argument we deduce that the probability there exists a pair of disjoint subsets of vertices of cardinality $r$ each, which does not satisfy property \textbf{Q\ref{i:split_bigsets}} with $K=\frac{n'}{r(1-2\varepsilon)}$ is at most
$$\binom{n'}{r}\cdot\binom{n'-r}{r}\cdot\exp\left(-\frac{n'\ln^{0.07}n'}{3}\right)\leq\exp\left(2r\ln\left(\frac{en'}{r}\right)-\frac{n'\ln^{0.07}n'}{3}\right)\leq\exp\left(-\frac{n'\ln^{0.07}n'}{4}\right).$$

Finally, note that $e(\hG)\sim\Bin(e(G'),\gamma)$ and thus $\Exp{e(\hG)}\leq (1+o(1))\frac{n\ln^{0.97}n}{2}$. Hence, using Theorem \ref{t:Chernoff} we deduce
$$\Prob{e(\hG)>2n'\ln^{0.97}n'}<\exp\left(-\frac{n\ln^{0.97}n}{6}\right).$$

Putting it all together we conclude that $\exp\left(-\frac{n\ln^{0.97}n}{6}\right)+\exp\left(-\frac{n'\ln^{0.07}n'}{4}\right)<2^{-n}$, and hence there exists a subgraph $\hG\subseteq G'$ with $e(\hG)\leq 2n'\ln^{0.97}n'$ and which satisfies the conditions of Theorem \ref{t:expgame}. It follows that $\hG\in\mM_{\mX_{R,c}}$ as claimed.
\end{proof}

\section{Hitting time of the $k$-vertex connectivity and perfect matching games}\label{s:HitTimePerfMatchKConn}
This short section is devoted to the proofs of Theorems \ref{t:kconn_game_hittime} and \ref{t:perf_match_game_hittime}. These two theorems are simple corollaries of the results presented in the previous sections.
\subsection{$k$-vertex connectivity}
As already mentioned in Section \ref{s:Preliminaries} we will provide a sufficient condition on $R$ and $c$ such that an $(R,c)$-expander will surely be $k$-vertex connected.
\begin{lem} \label{l:expansionToKconnectivity}
For every positive integer $k$, if $G=(V,E)$ is an $(R,c)$-expander such that $c\geq k$, and $Rc\geq\frac{1}{2}(|V|+k)$, then $G\in\VertConn{k}$.
\end{lem}
\begin{proof}
Assume for the sake of contradiction that there exists some set $S\subseteq V$ of size $k-1$ whose removal disconnects $G$. Denote the connected components of $G \setminus S$ by $S_1,\ldots,S_t$, where $t\geq 2$ and $1\leq|S_1|\leq\ldots\leq|S_t|$. If $|S_1|\leq R$, then $k-1=|S|\geq|N_G(S_1)|\geq c|S_1| \geq c\geq k$, which is clearly a contradiction. Assume then that $|S_1|>R$. For $i\in\{1,2\}$, let $A_i\subseteq S_i$ be an arbitrary subset of size $R$. It follows that $|V|\geq|S_1\cup S_2\cup N_G(S_1)\cup N_G(S_2)|\geq|N_G(A_1)\cup N_G(A_2)|=|N_G(A_1)|+|N_G(A_2)|-|N_G(A_1)\cap N_G(A_2)|\geq 2Rc-|S|\geq |V|+1$, which is clearly a contradiction. It follows that $G$ is $k$-vertex-connected as claimed.
\end{proof}
In order to prove Theorem \ref{t:kconn_game_hittime} it thus suffices to show that w.h.p. at the moment the random graph process first reaches minimum degree $2k$, Maker has a winning strategy for the $(R,c)$-expander game for suitably chosen values of $R$ and $c$. In doing so we will heavily rely on Theorem \ref{t:expgame}.
\begin{proof}[Proof of Theorem \ref{t:kconn_game_hittime}]
Fix some positive integer $k\geq 1$ and let $\tG=\{G_i\}_{i=0}^{\binom{n}{2}}$ denote the random graph process. Set $M=\tau(\tG;\delta_{2k})$, let $G=G_M$, $\Small=\mD_{\ln^{0.9}n}(G)$, $G'=G[V\setminus\Small]$ and denote by $n'$ the number of vertices in $G'$. Setting $c=k+2$, and $R=\frac{n'}{k+4}$, the conditions of Lemma \ref{l:splitGNM_noSMALL} are met, and thus $G'\in\mM_{\mX_{\frac{n'}{k+4},k+2}}$.

Maker's strategy is quite natural. He splits the board into $F_1=E(G')$ and $F_2=E_G(\Small,V\setminus \Small)$, and plays the corresponding two games in parallel, that is, in each move Maker will claim an edge of the board Breaker chose his last edge from (except for possibly his last move in one of the two games). Playing on the edges of $F_1$, Maker aims to build an $(\frac{n'}{k+4},k+2)$-expander. As noted above, Maker has a winning strategy for this game. Playing on the edges of $F_2$, Maker follows a simple pairing strategy which guarantees that, by the end of the game, the graph $H$ which Maker constructs will satisfy $d_H(v)\geq\lfloor d_G(v)/2\rfloor$ for every $v\in\Small$. To achieve this goal, whenever Breaker claims an edge which is incident with some vertex $v\in\Small$, Maker responds by claiming a different edge incident with $v$ if such an edge exists, and otherwise he claims an arbitrary free edge of $F_1 \cup F_2$. Since the minimum degree in $G$ is $2k$, it follows by Maker's strategy for the game on $F_2$ and by Claim \ref{c:Small-properties}, that in Maker's graph $H$, the vertices of $\Small$ form an independent set with $k$ edges emitting out of each vertex. Since the graph $H'=H[V\setminus\Small]$ is an $(\frac{n'}{k+4},k+2)$-expander, and since $(k+2)\cdot\frac{n'}{k+4}\geq\frac{1}{2}(n+k)$ holds for every $k\geq 1$ by Claim~\ref{c:Small-size}, Lemma \ref{l:expansionToKconnectivity} implies that $H'\in\VertConn{k}$. Adding to $H'$ the vertices of $\Small$ with their incident edges clearly keeps the $k$-vertex connectivity property, as connecting a new vertex to at least $k$ vertices of a $k$-vertex connected graph produces a $k$-vertex connected graph. This concludes the proof of the theorem.
\end{proof}

\subsection{Perfect matching}
Next, in order to show that expansion entails admitting a perfect matching, we make use of the well-known Berge-Tutte formula for the size of a maximum matching in a graph (see e.g. \cite[Corollary 3.3.7]{Wes2001}).
\begin{thm}[Berge-Tutte]\label{t:BerTut58} The maximum number of vertices which are saturated by a matching in a graph $G=(V,E)$ is $\min_{S\subseteq V}\left\{|V|+|S|-o(G-S)\right\}$.
\end{thm}
The following lemma is applicable regardless of the parity of the number of vertices in the graph.
\begin{lem} \label{l:expansionToMatching}
If $G=(V,E)$ is an $(R,c)$-expander such that $c\geq 2$ and
$(c+1)R\leq|V|<2Rc-8c$, then $G\in\PerfMatch$.
\end{lem}
\begin{proof}
From the conditions on $R$ and $c$ it follows that $Rc>|V|/2$ and, combined with $G$ being an $(R,c)$-expander, this trivially implies that the graph $G$ must be connected. Setting $S=\emptyset$, we have that $o(G-S)=1$ for odd $|V|$, and that $o(G-S)=0$ for even $|V|$. By Theorem~\ref{t:BerTut58} we can thus assume that $S\neq\emptyset$. We will in fact prove that  $|S| \geq c(G-S)$ holds for every non-empty $S \subseteq V$. It clearly suffices to prove this for every $\emptyset\neq S\subseteq V$ of cardinality $|S| \leq |V|/2$. Let $S$ be such a set, let $t=c(G-S)$, and let $S_1,\ldots, S_t$ denote the connected components of $G-S$, where $1\leq|S_1|\leq\ldots\leq|S_t|$. Assume first that there exists a set $A \subseteq \{1,\ldots,t\}$ such that $|S|/c<\left|\bigcup_{i \in A} S_i\right|\leq R$. By definition we have $N_G(\bigcup_{i\in A}S_i)\subseteq S$. It follows that $|S|\geq|N_G(\bigcup_{i \in A} S_i)|\geq c\left|\bigcup_{i \in A} S_i\right|>|S|$, which is clearly a contradiction. Hence, no such $A \subseteq \{1, \ldots, t\}$ exists. It follows that there must exist some $0 \leq j^* \leq t$ such that $\sum_{i=1}^{j^*} |S_i|\leq\lfloor |S|/c\rfloor$ and $|S_i|>R-|S|/c$ for every $j^*<i\leq t$. If $j^*\geq t-1$, then, since $|S_i|\geq 1$ for every $1\leq i\leq t$, it follows that $t\leq\sum_{i=1}^{t-1}|S_i|+1\leq\lfloor|S|/c\rfloor+1 \leq|S|$. Hence, we can assume that $j^*\leq t-2$. We claim that, under this assumption, $|S|\geq\frac{Rc}{2}$. Indeed, assume for the sake of contradiction that $1\leq|S|<\frac{Rc}{2}$ or equivalently, that $c(R-|S|/c)>|S|$. If $R-|S|/c\leq|S_{j^*+1}|\leq R$, then, as $S\supseteq N_{G}(S_{j^*+1})$ we have that $|S|\geq |N_{G}(S_{j^*+1})|\geq c(R-|S|/c)>|S|$, a contradiction. Therefore, $|S_i|>R$ for every $j^*<i\leq t$. Since $j^*\leq t-2$, for $i\in\{t-1,t\}$, we can choose $A_i\subseteq S_i$ to be an arbitrary subset of size $R$. It follows that $|V| \geq |S_{t-1}\cup S_t\cup N_G(S_{t-1})\cup N_G(S_t)|\geq |N_G(A_{t-1})\cup N_G(A_t)|=|N_G(A_{t-1})|+|N_G(A_t)|-|N_G(A_{t-1})\cap N_G(A_t)|\geq 2Rc-|S|>|V|$, which is, again, clearly a contradiction. We deduce that $|V|/4<Rc/2\leq|S|\leq|V|/2<Rc$. Note that under our assumption on $R$ and $c$ we have that $R-|S|/c>4$, and therefore $|S_i|\geq 5$ for every $j^*<i\leq t$. Moreover, since $|S_i| \geq 1$ holds for every $1 \leq i \leq j^*$, it follows that $j^*\leq |S|/c$. Putting everything together we have that $|S|>\frac{1}{3}\sum_{i=1}^t |S_i|\geq \frac{j^*+(t-j^*)(R-|S|/c)}{3}\geq \frac{5t-4j^*}{3}$, and therefore $t<\frac{|S|}{5}(3+\frac{4}{c})\leq |S|$. This concludes the proof of the lemma.
\end{proof}
In order to prove Theorem \ref{t:perf_match_game_hittime} we proceed
very similarly to the proof of Theorem \ref{t:kconn_game_hittime}.
\begin{proof}[Proof of Theorem \ref{t:perf_match_game_hittime}]
Let $\tG=\{G_i\}_{i=0}^{\binom{n}{2}}$ denote the random graph process. Set $M=\tau(\tG;\delta_{2})$, let $G=G_M$, $\Small=\mD_{\ln^{0.9}n}(G)$, $G'=G[V\setminus\Small]$ and denote by $n'$ the number of vertices in $G'$. Setting $c=8$, and $R=\frac{n'}{10}$, the conditions of Lemma \ref{l:splitGNM_noSMALL} are met, and thus $G'\in\mM_{\mX_{\frac{n'}{10},8}}$.

Maker's strategy is quite natural. He splits the board into $F_1=E(G')$ and $F_2=E_G(\Small,V\setminus \Small)$, and plays the corresponding two games in parallel, that is, in each move Maker will claim an edge of the board Breaker chose his last edge from (except for possibly his last move in one of the two games). Playing on the edges of $F_1$, Maker aims to build an $(n'/10,8)$-expander. As noted above, Maker has a winning strategy for this game. We denote the restriction of the graph built by Maker by the end of the game to the edges of $F_1$ by $H_1$. Playing on the edges of $F_2$, Maker follows a simple pairing strategy which guarantees that, by the end of the game, the graph $H_2$ which Maker constructs will satisfy $d_{H_2}(v)\geq\lfloor d_G(v)/2\rfloor$ for every $v\in\Small$. To achieve this goal, whenever Breaker claims an edge which is incident with some vertex $v\in\Small$, Maker responds by claiming a different edge incident with $v$ if such an edge exists, and otherwise he claims an arbitrary free edge of $F_1\cup F_2$. Recalling Claim \ref{c:Small-properties} we can assume that $\Small$ is an independent set in $G$ and that no two vertices in $\Small$ share a common neighbor. As the minimum degree in $G$ is $2$, Maker's graph, $H=H_1\cup H_2$, will contain at least one edge emitting out of every vertex in $\Small$, each incident with a different vertex of $V\setminus\Small$. Therefore, there exists a matching in $\mM$ which covers all vertices of $\Small$. Let $T$ denote the set of vertices of $V \setminus \Small$ which are covered by $\mM$. Again, by Claim \ref{c:Small-properties} we can assume that no two vertices in $T$ share a common neighbor (as this would create a path of length $4$ between two vertices in $\Small$). Since, the graph $H_2$ is an $(n'/10,8)$-expander, it follows by Claim \ref{c:indsetremovalfromexpander} that the graph $H' = H_2 \setminus T$ is an $(n'/10,7)$-expander. The values $R=n'/10$ and $c=7$ satisfy the condition of Lemma \ref{l:expansionToMatching}, implying that $H'\in\PerfMatch$. Let $\mM'$ be some perfect matching of $H'$, then $\mM \cup \mM'$ is a perfect matching of $H$. This concludes the proof of the theorem.
\end{proof}

\section{Hitting time of the Hamiltonicity game}\label{s:HitTimeHam}
Our proof of Theorem \ref{t:hamilton_game_hittime} is fairly similar to the two proofs presented in the previous section. However, having built an appropriate expander, Maker will need to claim additional edges in order to transform his expander into a Hamiltonian graph. In order to describe the relevant connection between Hamiltonicity and $(R,c)$-expanders, we require the notion of \emph{boosters}.
\begin{defn}\label{d:booster}
For every graph $G$, we say that a non-edge $\{u,v\}\notin E(G)$ is a \emph{booster} with respect to $G$, if either $G \cup \{u,v\}$ is Hamiltonian or $\ell(G \cup \{u,v\})>\ell(G)$. We denote by $\mB_G$ the set of boosters with respect to $G$.
\end{defn}
The following is a well-known property of $(R,2)$-expanders (see e.g. \cite{FriKri2008}).
\begin{lem}\label{l:expander_booster}
If $G$ is a connected non-Hamiltonian $(R,2)$-expander, then $|\mB_G|\geq R^2/2$.
\end{lem}

Our goal is to show that during a game on an appropriate graph $G$, assuming Maker can build a subgraph of $G$ which is an $(R,c)$-expander, he can also claim sufficiently many such boosters, so that his $(R,c)$-expander becomes Hamiltonian. In order to do so, we further analyze the structure of the random graph process.
\begin{lem}\label{l:GNM_NoGammaDestroy}
If $\tG=\{G_i\}_{i=0}^{\binom{n}{2}}$ is the random graph process and $M=\tau(\tG;\delta_4)$, then w.h.p. $G_M$ does not contain a connected non-Hamiltonian $(n/5,2)$-expander $\Gamma$ with at most $n\ln^{0.98}n$ edges such that $|E(G_M)\cap\mB_{\Gamma}|\leq \frac{n\ln n}{100}$.
\end{lem}
\begin{proof}
First we note that any $(n/5,2)$-expander must be connected, as each connected component must be of size at least $n/5 + 2n/5 > n/2$. Let $m_4\leq M'\leq M_4$ be an integer, let $p=M'/\binom{n}{2}>\frac{\ln n}{n}$, and let $G=(V,E) \sim \GNP$. Our goal is to prove that the probability that $G$ contains a connected non-Hamiltonian $(n/5,2)$-expander subgraph $\Gamma$ with at most $n\ln^{0.98}n$ edges such that $|E \cap \mB_{\Gamma}|\leq n\ln^{0.98}n$ is ``much smaller'' than the probability that $e(G)=M'$. Summing over all integral values of $M'$ in the interval $[m_4, M_4]$, and applying Claim \ref{c:GNPtoGNM} to each of these values, will enable us to complete the proof.

Let $\mS$ denote the set of all labeled non-Hamiltonian $(n/5,2)$-expanders on the vertex set $V$ which have at most $n\ln^{0.98}n$ edges. Fix a graph $\Gamma=(V,F)\in\mS$, then clearly $\Prob{\Gamma\subseteq G}=p^{|F|}$. Now, let $G'=(V,E\setminus F)\sim\GNP_{-F}$. By definition, every booster with respect to $\Gamma$ is a non-edge in $\Gamma$, hence $\mB_\Gamma$ is a subset of the potential pairs of the graph $G'$. Lemma \ref{l:expander_booster} implies that $|\mB_{\Gamma}|\geq n^2/50$, and since $|E(G')\cap\mB_{\Gamma}|\sim\Bin(|\mB_{\Gamma}|,p)$, it follows that $\Exp{|E(G')\cap\mB_{\Gamma}|}\geq\frac{n^2p}{50}>\frac{n\ln n}{50}$. Applying Theorem \ref{t:Chernoff} we have
\begin{equation*}
\Prob{|E(G')\cap\mB_{\Gamma}|\leq\frac{n\ln n}{100}}\leq\exp\left(-\frac{\left(1-\frac{50}{100}\right)^2n^2p}{100}\right)=\exp\left(-\frac{n^2p}{400}\right).
\end{equation*}
Next, we note that by the independence of appearance of edges in $\GNP$, the event $\Gamma\subseteq G$ and the event that some booster $e$ with respect to $\Gamma$ was chosen among the edges of $G'$, are independent events. We can thus use a union bound argument by going over all $\Gamma\in\mS$ to upper bound the probability that $G$ contains a connected non-Hamiltonian $(n/5,2)$-expander $\Gamma$ with at most $n\ln^{0.98}n$ edges, such that $|E \cap \mB_{\Gamma}|\leq\frac{n\ln n}{100}$ as follows
\begin{eqnarray*}
&&\sum_{m=1}^{n\ln^{0.98}n}\binom{\binom{n}{2}}{m}p^m\cdot\exp\left(-\frac{n^2p}{400}\right)\\
&\leq& \sum_{m=1}^{n\ln^{0.98}n}\left(\frac{e n^2 p}{2m}\right)^m \cdot\exp\left(-\frac{n^2p}{400}\right)\\
&\leq& \sum_{m=1}^{n\ln^{0.98}n}\exp\left(m\cdot\left(1+\ln\left(\frac{n^2p}{2m}\right)\right)-\frac{n^2p}{400}\right)\\
&\leq& \exp\left(-\frac{n^2p}{401}\right).
\end{eqnarray*}
Using Claim \ref{c:GNPtoGNM}, the above calculation implies that the same event, with $G\sim\GNp{M'}$, is upper bounded by $\sqrt{2\pi M'} \cdot \exp\left(-\frac{n^2p}{401}\right)  \leq \exp \left(\frac{-n\ln n}{401}\right)$. Taking the union bound over all integral values of $m_4\leq M'\leq M_4$, we conclude that the probability there exists such an integer $M'$ for which $G_{M'}$ violates the claim is at most $(M_4 - m_4 + 1) \cdot \exp \left(\frac{-n\ln n}{401}\right) \leq n\ln\ln\ln n \cdot \exp \left(\frac{-n\ln n}{401}\right) = o(1)$.
\end{proof}

We are now ready to present the full proof of Theorem \ref{t:hamilton_game_hittime}.
\begin{proof}[Proof of Theorem \ref{t:hamilton_game_hittime}]
Let $\tG=\{G_i\}_{i=0}^{\binom{n}{2}}$ denote the random graph process. Set $M=\tau(\tG;\delta_{4})$, let $G=G_M$, $\Small=\mD_{\ln^{0.9}n}(G)$, $G'=G[V\setminus\Small]$ and denote by $n'$ the number of vertices in $G'$. By Claim~\ref{c:Small-size} we can assume that $|\Small| \leq n^{0.3}$. Setting $c=3$, and $R=\frac{9n'}{40}$, the conditions of Lemma \ref{l:splitGNM_noSMALL} are met, and thus there exists a subgraph $\hG\subseteq G'$ such that $\hG\in\mM_{\mX_{\frac{9n'}{40},3}}$ and $e(\hG)\leq 2n'\ln^{0.97}n'$.

Maker's strategy is quite natural. It consists of two phases. Let $e_i$ denote the edge selected by Maker in his $i$th move and let $H_i=(V,\{e_1,\ldots,e_i\})$ denote the graph Maker has built during his first $i$ moves. Let $H'$ denote Maker's graph at the end of the first phase and let $H$ denote Maker's graph at the end of the second phase, that is, Maker's final graph. Before the game starts, Maker splits the board $E(G)$ into three parts $F_1=E(\hG)$, $F_2=E_G(\Small,V\setminus \Small)$ and $F_3=E(G'\setminus \hG)$. During the first phase, Maker plays two games in parallel, one on $F_1$ and the other on $F_2$. For every $j \geq 1$, on his $j$th move of the first phase, Maker claims an edge of $F_1 \cup F_2$, according to his strategy for each of the two games. If on his $j$th move Breaker claims an edge of $F_i$, for some $i \in \{1,2\}$, then Maker claims an edge of $F_i$ as well (unless he has already achieved his goal in the game on $F_i$). If Breaker claims an edge of $F_3$, them Maker claims an edge of $F_1 \cup F_2$. Playing on the edges of $F_1$, Maker aims to build a $(9n'/40,3)$-expander, $H'_1$. As noted above, Maker has a winning strategy for this game. Moreover, since $|F_1| \leq 2n'\ln^{0.97}n'$, Maker can build such an expander within at most $t_{1,1} := n'\ln^{0.97}n'$ moves. Playing on the edges of $F_2$, Maker follows a simple pairing strategy which guarantees that, by the end of the game, the graph, $H'_2$, which Maker constructs, will satisfy $d_{H'_2}(v)\geq 2$ for every $v\in\Small$. To achieve this goal, whenever Breaker claims an edge which is incident with some vertex $v\in\Small$, Maker responds by claiming a different edge incident with $v$, unless his current graph already contains two edges which are incident with $v$ in which case he claims another free edge of $F_1 \cup F_2$ which brings him closer to his goal in the corresponding game. Hence, the number of moves required for Maker to reach his goal in the game on $F_2$ is at most $t_{1,2} := 2|\Small| \leq 2n^{0.3}$. It follows by Claim \ref{c:Small-properties} that $\Small$ is an independent set and that no two edges emitting from $\Small$ are incident with the same vertex of $V \setminus \Small$. Hence, Maker's graph, $H'_2$, satisfies $N_{H'_2}(U')\geq 2|U'|$ for every $U'\subseteq\Small$. Applying Claim \ref{c:indsetaddtoexpander} and noting that $9n'/40 \geq n/5$, it follows that $H'=H'_1\cup H'_2$ is an $(n/5,2)$-expander. Clearly, Maker's final graph $H$ is an $(n/5,2)$-expander as well. A crucial point to keep in mind is that the number of moves required for Maker to construct his $(n/5,2)$-expander $H'$, is $t_1 = t_{1,1}+t_{1,2} = o(n\ln^{0.98}n)$.

After having completed the construction of $H'$, Maker proceeds to the second phase of his strategy. Let $t_2 \leq n$ denote the number of moves Maker plays during the second phase. For every $t_1 < j \leq t_1 + t_2$, on his $j$th move, Maker claims an edge of $G$ which is a booster with respect to $H_{j-1}$. This is possible since, throughout the game Breaker claims at most $t_1 + t_2 \leq t_1 + n$ edges of $G$, but by Lemma \ref{l:GNM_NoGammaDestroy}, w.h.p. either $H_{j-1}$ is Hamiltonian or it has at least $n \ln n/100 > t_1 + n$ boosters among the edges of $G$. It follows by the definition of a booster that either $H_j$ is Hamiltonian or $\ell(H_j) > \ell(H_{j-1})$. Repeating the same argument $t_2 \leq n$ times, we conclude that $H$ is Hamiltonian as claimed.
\end{proof}

\section{Remarks on possible generalizations}\label{s:Generalizations}
We note that, by using a slight modification of our proofs, Theorems \ref{t:perf_match_game_hittime} and \ref{t:hamilton_game_hittime} can in fact be extended. For every positive integer $k\geq 1$, let $\PerfMatch^k$ and $\Hamiltonicity^k$ denote the graph properties of admitting $k$ pairwise edge-disjoint perfect matchings, and $k$ pairwise edge-disjoint Hamilton cycles respectively.
\begin{our_thm}\label{t:k-perf_match_game_hittime}
For every fixed integer $k\geq 1$, if $\tG$ is the random graph process, then w.h.p.
$$\tau(\tG;\mM_{\PerfMatch^k})=\tau(\tG;\delta_{2k}).$$
\end{our_thm}
\begin{our_thm}\label{t:k-hamilton_game_hittime}
For every fixed integer $k\geq 1$, if $\tG$ is the random graph process, then w.h.p.
$$\tau(\tG;\mM_{\Hamiltonicity^k})=\tau(\tG;\delta_{4k}).$$
\end{our_thm}
Theorem \ref{t:k-hamilton_game_hittime} can be viewed as a Combinatorial game analog of the classical result of Bollob\'{a}s and Frieze \cite{BolFri85} who proved that w.h.p. $\tau(\tG;\Hamiltonicity^k)=\tau(\tG;\delta_{2k})$ (see also \cite{FriKri2008} for an extension to non-constant minimum degree in the $\GNP$ model).

We now sketch how the proof of Theorem \ref{t:hamilton_game_hittime} can be adapted so as to entail Theorem \ref{t:k-hamilton_game_hittime}. Similarly, the proof of Theorem \ref{t:k-perf_match_game_hittime} can be obtained using appropriate modifications to the proof of Theorem \ref{t:perf_match_game_hittime}, but as this case is simpler, we omit the details.

It suffices to prove that when removing all vertices of degree at most $\ln^{0.9}n$ from the random graph $\GNM$, where $M = \tau(\tG; \delta_{4k})$, playing on this subgraph $G'$ on $n'$ vertices, w.h.p. Maker can quickly (that is, within $o(n' \ln n')$ moves) build an $(9n'/40k,3k)$-expander $H'$ for which the property \textbf{M\ref{i:rcexpander_suffcondition_bigsets}} with $r=n'/ln^{0.4}n'$ holds. Moreover, at the same time, Maker can ensure that the minimum degree of his graph will be at least $2k$. After the removal of $0\leq i\leq k-1$ edge-disjoint Hamilton cycles from the original graph we have removed a $2i$-regular graph from from $H'$ and are left with a graph $\hH_i$ (which is spanned by the vertices which are not in $\Small$) for which $|N_{\hH_i}(U)|\geq 3k|U|-2i|U|\geq (k+2)|U|$ for every $U\subseteq V(H')$ of cardinality $|U|\leq 9n'/40k$. To complete the proof it is left to note that for the choice of the parameter $r$ guarantees that between sets of linear size there is a super-linear number of edges. It is not hard to see that adding back the vertices of $\Small$ who are all incident to at least $2k-2i\geq 2$ edges results in a connected $(n/5,2)$-expander. This graph has many boosters which Breaker could not have taken them all, and  Maker can thus continue playing for another Hamilton cycle using the boosters left in the graph. As there is a super-linear number of boosters and Breaker can claim at most $n$ of them per Hamilton cycle, Maker can keep playing this way until he completely saturates his vertices of minimum degree.

\subsection*{Acknowledgments}
This research was partially conducted while the authors were present (as guests or members) at the Institute of Theoretical Computer Science at ETH Z\"{u}rich. We would like to thank Angelika Steger and her group for the support and wonderful facilities provided during this time.
\bibliographystyle{abbrv}
\bibliography{CombGames}

\begin{thebibliography}{10}

\bibitem{AloSpe2008}
N.~Alon and J.~H. Spencer.
\newblock {\em The Probabilistic Method}.
\newblock Wiley-Interscience Series in Discrete Mathematics and Optimization.
  John Wiley \& Sons, third edition, 2008.

\bibitem{Bec81}
J.~Beck.
\newblock On positional games.
\newblock {\em Journal of Combinatorial Theory, Series {A}}, 30(2):117--133,
  1981.

\bibitem{Bec2008}
J.~Beck.
\newblock {\em Combinatorial Games: {T}ic-{T}ac-{T}oe theory}.
\newblock Cambridge University Press, New York, 2008.

\bibitem{BenKriSudPre}
S.~Ben-Shimon, M.~Krivelevich, and B.~Sudakov.
\newblock Local resilience and {H}amiltonicity {M}aker-{B}reaker games in
  random regular graphs.
\newblock {\em Combinatorics, Probability, and Computing}, to appear.

\bibitem{Bol84}
B.~Bollob\'{a}s.
\newblock The evolution of sparse graphs.
\newblock In B.~Bollob\'{a}s, editor, {\em Porceedings of Cambridge
  Combinatorial conference in honor of Paul Erd\H{o}s}, Graph Theory and
  Combinatorics, pages 35--57. Academic Press, 1984.

\bibitem{Bol2001}
B.~Bollob\'{a}s.
\newblock {\em Random Graphs}.
\newblock Cambridge University Press, 2001.

\bibitem{BolFri85}
B.~Bollob\'{a}s and A.~Frieze.
\newblock On matchings and hamiltonian cycles in random graphs.
\newblock In {\em Random Graphs (Pozna\'{n} 1983)}, volume~28 of {\em Annals of
  Discrete Mathematics}, pages 23--46. North-Holland, Amsterdam, 1985.

\bibitem{BolTho85}
B.~Bollob\'{a}s and A.~G. Thomason.
\newblock Random graphs of small order.
\newblock In {\em Random Graphs (Pozna\'{n} 1983)}, volume~28 of {\em Annals of
  Discrete Mathematics}, pages 47--97. North-Holland, Amsterdam, 1985.

\bibitem{ChvErd78}
V.~Chv\'{a}tal and P.~Erd\H{o}s.
\newblock Biased positional games.
\newblock In {\em Algorithmic aspects of combinatorics (Vancouver 1976)},
  volume~2 of {\em Annals of Discrete Mathematics}, pages 221--229. 1978.

\bibitem{ErdSel73}
P.~Erd\H{o}s and J.~Selfridge.
\newblock On a combinatorial game.
\newblock {\em Journal of Combinatorial Theory, Series {A}}, 14:298--301, 1973.

\bibitem{FelKri2008}
O.~N. Feldheim and M.~Krivelevich.
\newblock Winning fast in sparse graph construction games.
\newblock {\em Combinatorics, Probability and Computing}, 17(6):781--791, 2008.

\bibitem{FriKri2008}
A.~Frieze and M.~Krivelevich.
\newblock On two {H}amilton cycle problems in random graphs.
\newblock {\em Israel Journal of Mathematics}, 166:221--234, 2008.

\bibitem{HefEtAl2009b}
D.~Hefetz, M.~Krivelevich, M.~Stojakovi\'{c}, and T.~Szab\'{o}.
\newblock Fast winning strategies in {M}aker-{B}reaker games.
\newblock {\em Journal of Combinatorial Theory, Series {B}}, 99(1):39--47,
  2009.

\bibitem{HefEtAl2009}
D.~Hefetz, M.~Krivelevich, M.~Stojakovi\'{c}, and T.~Szab\'{o}.
\newblock A sharp threshold for the {H}amilton cycle {M}aker-{B}reaker game.
\newblock {\em Random Structures and Algorithms}, 34(1):112--122, 2009.

\bibitem{HefEtAlPre}
D.~Hefetz, M.~Krivelevich, M.~Stojakovi\'{c}, and T.~Szab\'{o}.
\newblock Global {M}aker-{B}reaker games on sparse graphs.
\newblock preprint.

\bibitem{HefSti2009}
D.~Hefetz and S.~Stich.
\newblock On two problems regarding the {H}amilton cycle game.
\newblock {\em The Electronic Journal of Combinatorics}, 16(1):R28, 2009.

\bibitem{JanLucRuc2000}
S.~Janson, T.~{\L}uczak, and A.~Ruci{\'n}ski.
\newblock {\em Random Graphs}.
\newblock Wiley-Interscience Series in Discrete Mathematics and Optimization.
  John Wiley \& Sons, 2000.

\bibitem{KomSze83}
J.~Koml\'os and E.~Szemer\'edi.
\newblock Limit distributions for the existence of {H}amilton circuits in a
  random graph.
\newblock {\em Discrete Mathematics}, 43(1):55--63, 1983.

\bibitem{Leh64}
A.~Lehman.
\newblock A solution of the {S}hannon switching game.
\newblock {\em Journal of the Society for Industrial and Applied Mathematics},
  12(4):687--725, 1964.

\bibitem{PalSpe95}
E.~M. Palmer and J.~J. Spencer.
\newblock Hitting time for k edge-disjoint spanning trees in a random graph.
\newblock {\em Periodica Mathematica Hungarica}, 31(3):235--240, 1995.

\bibitem{Pek96}
A.~Peke\v{c}.
\newblock A winning strategy for the {R}amsey graph game.
\newblock {\em Combinatorics, Probability and Computing}, 5(3):267--276, 1996.

\bibitem{Sto2005}
M.~Stojakovi\'{c}.
\newblock {\em Games on Graphs}.
\newblock PhD thesis, ETH Z\"urich, 2005.

\bibitem{StoSza2005}
M.~Stojakovi\'c and T.~Szab\'o.
\newblock Positional games on random graphs.
\newblock {\em Random Structures and Algorithms}, 26(1-2):204--223, 2005.

\bibitem{Wes2001}
D.~B. West.
\newblock {\em Introduction to Graph Theory}.
\newblock Prentice Hall, 2001.

\end{thebibliography}
\end{document}